\documentclass[a4paper,10pt,reqno]{amsart}
\usepackage{amsmath}
\usepackage{amsthm,enumerate}

\usepackage{graphicx}
\usepackage{amssymb}

\usepackage{appendix}

\usepackage[italian,english]{babel}

\usepackage[utf8x]{inputenc}
\usepackage{fancyhdr}

\usepackage{calc}
\usepackage{url}

\usepackage[text={6in,8.6in},centering]{geometry}

\usepackage{srcltx}

\usepackage[T1]{fontenc}

\usepackage[usenames,dvipsnames]{color}

\makeatletter
\def\cleardoublepage{\clearpage\if@twoside \ifodd\c@page\else%
         \hbox{}%
     \thispagestyle{empty}
     \newpage%
     \if@twocolumn\hbox{}\newpage\fi\fi\fi}
\makeatother

\let\cleardoublepage\clearpage

\addto\captionsitalian{}

\newtheorem{thm}{Theorem}[section]
\newtheorem{cor}[thm]{Corollary}
\newtheorem{lem}[thm]{Lemma}
\newtheorem{pro}[thm]{Proposition}
\newtheorem{den}[thm]{Definition}
\newtheorem{oss}[thm]{Remark}
\numberwithin{equation}{section}

\begin{document}

\title[]{The porous medium equation with measure data \\ on negatively curved Riemannian manifolds}

\author {Gabriele Grillo, Matteo Muratori, Fabio Punzo}

\address {Gabriele Grillo: Dipartimento di Matematica, Politecnico di Milano, Piaz\-za Leonardo da Vinci 32, 20133 Milano, Italy}
\email{gabriele.grillo@polimi.it}

\address{Matteo Muratori: Dipartimento di Matematica ``F. Enriques'', Universit\`a degli Studi di Milano, via Cesare Saldini 50, 20133 Milano, Italy}
\email{matteo.muratori@unimi.it}

\address{Fabio Punzo: Dipartimento di Matematica ``F. Enriques'', Universit\`a degli Studi di Milano, via Cesare Saldini 50, 20133 Milano, Italy}
\email{fabio.punzo@unimi.it}

\keywords{Porous medium equation; Sobolev
inequalities; Green function; potential analysis; superharmonic functions; nonlinear diffusion equations; smoothing effect;
asymptotics of solutions.}
%
%
%
%


\begin{abstract}
We investigate existence and uniqueness of weak solutions of the Cauchy problem for the porous medium equation on negatively curved Riemannian
manifolds. We show existence of solutions taking as initial condition a finite Radon measure, not necessarily positive. We then establish uniqueness
in the class of nonnegative solutions, under a quadratic lower bound on the Ricci curvature. On the other hand, we prove that any weak solution of the porous medium equation necessarily takes on as initial
datum a finite Radon measure. In addition, we obtain some results in potential analysis on manifolds, concerning the validity of a modified
version of the mean-value inequality for superharmonic functions, and properties of potentials of positive Radon measures. Such results are new and of independent interest, and are crucial for our approach.
\end{abstract}
\maketitle


\section{Introduction}
We are concerned with existence and uniqueness of weak solutions
of Cauchy problems for the \emph{porous medium equation} on
Riemannian manifolds of the following type:
\begin{equation}\label{e64}
\begin{cases}
u_t \,=\, \Delta (u^m) & \textrm{in}\;\; M\times (0,\infty)\,, \\
u \,=\, \mu & \textrm{on}\;\; M\times \{0\}\,,
\end{cases}
\end{equation}
where $M$ is an $N$-dimensional complete, simply connected Riemannian manifold with nonpositive sectional curvatures (namely a \emph{Cartan-Hadamard manifold}), $\Delta$ is the Laplace-Beltrami
operator on $M$, $m>1$ and $\mu$ is a finite Radon measure on $M$. Note that, when dealing with changing-sign solutions, as usual we set $ u^m = |u|^{m-1}u $.


\smallskip

In the special case of the Euclidean space, problem \eqref{e64} has deeply been investigated in \cite{Pierre}. In particular, existence and
uniqueness results for nonnegative solutions have been established. More recently, we should mention that similar results have been generalized to the \emph{fractional} porous medium equation \cite{Vaz14,GMPmu,GMPas}. Furthermore, problem \eqref{e64} with the choice $M=\mathbb{H}^N$, namely
\begin{equation}\label{e64i}
\begin{cases}
u_t \,=\, \Delta (u^m) & \textrm{in}\;\; \mathbb{H}^N \times (0,\infty)\,, \\
u \,=\, \mu & \textrm{on}\;\; \mathbb{H}^N \times \{0\}\,,
\end{cases}
\end{equation}
where $ \mathbb{H}^N $ denotes the $N$-dimensional hyperbolic space, has lately been addressed in a number of papers. In fact, in \cite{VazH} it has
been studied for $m>1$ and $ \mu $ a Dirac delta, in \cite{Pu1,Pu2} for $m>1$ and $ \mu \in L^\infty(\mathbb{H}^N) $, and in \cite{GMhyp} for
$ \mu \in L^p(\mathbb{H}^N) $ for any $p>p_0$ (for a certain $p_0(m,N)$) in a {\it fast diffusion} regime, i.e.~$(N-2)/(N+2)<m<1$. More precisely, in
\cite{VazH} a thorough analysis on the fundamental solution of the differential equation in \eqref{e64i}, that is the solution of \eqref{e64i} with
$\mu=\delta$, is performed. Such special solution is then used to study the large-time behaviour of nonnegative solutions to \eqref{e64i} with $ \mu
\in L^1(\mathbb{H}^N) $.

\smallskip

The aim of our paper is to investigate existence and uniqueness of weak solutions to problem \eqref{e64}, under the hypothesis that the sectional curvatures are nonpositive (this is enough for existence), and that the Ricci curvature is bounded from below by $-C(1+\operatorname{dist}(x,o)^2) $ for some positive constant $C$ and a fixed point $ o \in M $ (this is required for uniqueness). Under our assumptions $M$ is necessarily {\it nonparabolic} (see Section \ref{RG}), hence the Green function $G(x,y)$ on $M$ is finite for all $x\neq y$.

In particular, we show that for any given finite Radon measure $\mu$ (not necessarily positive) there exists a weak solution to problem \eqref{e64}
which takes on the initial condition in a suitable ``dual'' sense. Note that, in general, such solution can change sign. On the other hand, we are
able to prove uniqueness under the additional assumption that $ \mu $, and so the corresponding solutions, is nonnegative. Furthermore, we show that
any weak solution of the differential equation in problem \eqref{e64} (i.e.~without a prescribed initial condition) necessarily takes on, in a
suitable weak sense, a finite Radon measure as $t\to 0^+$, which is uniquely determined (the initial trace). Observe that this property also justifies the fact that we
consider a finite Radon measure as initial datum in problem \eqref{e64}. Let us stress that no result in the literature seems to be available as
concerns {\em signed} measures, for which we can prove existence and trace results.
\smallskip

Let us mention that, in order to prove that the initial condition is taken on in a suitable weak sense, we
exploit some results from potential theory on Riemannian manifolds that we have established here precisely for this purpose, which also have an
independent interest. To be specific, we extend to Riemannian manifolds some results for potentials of nonnegative measures given in the monograph
\cite{Land}, and we obtain a suitable \emph{mean-value inequality} for superharmonic (and subharmonic) functions, without assuming any sign condition
and in particular dealing also with positive superharmonic functions. Note that, in contrast with the classical results in \cite{LS}, where the
standard mean value of nonnegative smooth subharmonic functions are considered, we deal with a modified mean value which takes into account the Green
function of $ (-\Delta) $ on $M$: this allows us to remove the nonnegativity assumption. This is essential for our purposes; in fact, since we deal
with positive superharmonic functions, the results in \cite{LS} cannot be applied in such case. In addition, we work with lower semi-continuous
functions with values in $(-\infty,\infty]$ which are superharmonic (or subharmonic with values in $[-\infty,\infty)$) in a distributional sense
only: in fact we shall apply such inequalities to potentials of Radon measures. In establishing such modified mean-value inequalities, we follow the line of arguments of \cite{BonLan} (see also \cite{FG} and references therein), where similar results are obtained in Euclidean space for general second-order elliptic operators.

Note that mean-value inequalities, involving Green functions, in the context of general strongly nonparabolic Riemannian manifolds, have also been first proved in \cite{N}. However, such inequalities are established for smooth functions, although they can be weakened to hold for Lipschitz functions (see Remark 2.4 in \cite{N}), a class of functions which is not sufficient for our purposes.


We remark that the above mentioned results in potential analysis will be crucial also in the proof of uniqueness. In fact, by adapting to the present
setting the general ``duality method'' (see \cite{Pierre}), we consider the problem satisfied by the difference of the potentials of any two
solutions taking on the same initial measure, and the corresponding dual one.

Let us also mention that, in a different framework, the use of Green functions in connection with the porous medium equation has recently been
performed in \cite{BV}, to obtain certain sharp priori estimates.

\smallskip
From a general viewpoint, the fact that we are considering non-positively curved Riemannian manifolds implies relevant differences with respect to
the Euclidean space, which is a particular case. In fact, in view of our hypotheses on sectional curvatures, we could have different properties for
the Green function and for the growth of the volume of balls (which can be exponential with respect to the radius, as in $\mathbb H^N$, or even faster).
Therefore, we need to use more delicate cut-off arguments which exploit crucial integrability properties of the Green function. In addition, our
assumption concerning the bound from below for the Ricci curvature (see \eqref{H}-(ii) below) is essential since it ensures conservation of mass for the aforementioned dual problem, a key tool in the uniqueness proof. It is not surprising that such bound on the Ricci curvature is essential for uniqueness, since it implies \emph{stochastic completeness} of $M$, which is equivalent to uniqueness of bounded solutions in the linear case (i.e.~for the heat equation), such a condition being sharp for stochastic completeness, see \cite{Grig}.

The potential techniques we exploit allow us to establish an identity which expresses the Green function in terms of the time integral of the
solution of problem \eqref{e64} with $\mu=\delta_{x_0}$ for any $x_0\in M$. Such formula holds, indeed, on general Riemannian manifolds, without
specific assumptions on its curvatures. In particular, it seems to be new, to our knowledge, even in the Euclidean framework. On the other hand, it extends to the
nonlinear case a well-known formula, which relates the Green function to the heat kernel. This result implies in particular that a manifold is nonparabolic
if and only if the Barenblatt solution is integrable in time. We are not aware of previous results connecting nonparabolicity of a manifold to properties of nonlinear evolutions of the kind studied here.

\smallskip

The paper is organized as follows. In Section \ref{Stat} we state the main results and we give the precise definition of solution to problem
\eqref{e64}. In Section \ref{RG} we recall some useful preliminaries in Riemannian geometry and basic facts concerning analysis on manifolds. Then in
Section \ref{Pot} we obtain some results in potential analysis on manifolds; although they are mostly used in the subsequent sections, they also have
an independent interest. Existence of solutions is shown in Section \ref{sec: exi-proof}, along with the integral identity involving the Green function. Finally, in Section \ref{proof-uniq} we prove both uniqueness of solutions and the results concerning the initial trace.
\medskip

We thank the referees of this paper for their careful reading of the original version of this manuscript, and for several comments which allowed us to strengthen some of our results.

\begin{oss}
\rm Our results are presented for simplicity in the case of Cartan-Hadamard manifolds of dimension $ N \ge 3 $. However, they hold with identical proofs under the following more general assumptions:
\begin{itemize}
\item $M$ is nonparabolic, complete and noncompact. Moreover, it supports the Sobolev-type inequality $\| f \|_{2\sigma}\le C\,\|\nabla f \|_2$  for some $ \sigma>1 $, $C>0$ and all $ f \in C^\infty_c(M)$;
\item $G(x,y)\to0$ as dist\,$(x,y)\to+\infty$, uniformly in $x\in K$, given any compact set $K\subset M$;
\item there exists $ o \in M $ such that $ x \mapsto \operatorname{dist}(x,o) $ is $C^2(M\setminus B)$ for some neighbourhood $B$ of $o$ and $|\Delta_x\operatorname{dist}(x,o)|\le c \operatorname{dist}(x,o) $ for a suitable constant $c>0$ and $ \operatorname{dist}(x,o) $ large (not necessary for existence).
\end{itemize}
Note that the above properties are fulfilled if $M$ is, for instance, a nonparabolic, complete and noncompact Riemannian manifold of dimension $ N \ge 3 $ possessing a pole $o$ such that $\operatorname{cut}(o)=\emptyset$ (i.e.~the \emph{cut locus} at $o$ is empty) and assumption \eqref{H}-(ii) below holds, with nonpositive sectional curvatures outside a compact set.
\end{oss}

\section{Statements of the main results}\label{Stat}

We consider {\it Cartan-Hadamard} manifolds, i.e.~complete, noncompact, simply connected Riemannian manifolds with nonpositive
sectional curvatures. Observe that (see e.g.~\cite{Grig, Grig3}) on Cartan-Hadamard manifolds the {\it cut locus} of any point $o$ is empty.
So, for any $x\in M\setminus \{o\}$, one can define its {\it polar coordinates} with pole at $o$. Namely, for any point $x\in M\setminus \{o\}$ there
exists a polar radius $\rho(x) := d(x, o)$ and a polar angle $\theta\in \mathbb S^{N-1}$ such that the geodesics from $o$ to $x$ starts
at $o$ with direction $\theta$ in the tangent space $T_oM$ (and has length $\rho$). Since we can identify $T_o M$ with $\mathbb R^{N}$, $\theta$ can be regarded as a point of $\mathbb S^{N-1}:=\{x\in \mathbb R^{N}:\,|x|=1\}.$

The Riemannian metric in $M\setminus \{o\}$ in polar coordinates reads
\[ds^2 = d\rho^2+A_{ij}(\rho, \theta)d\theta^i d\theta^j,\]
where $(\theta^1, \ldots, \theta^{N-1})$ are coordinates in $\mathbb S^{N-1}$ and $(A_{ij})$ is a positive definite matrix.

Let
$$ \mathcal A:=\left\{f\in C^\infty((0,\infty))\cap C^1([0,\infty)): \, f'(0)=1, \, f(0)=0, \, f>0 \ \textrm{in}\;\, (0,\infty)\right\} . $$
We say that $M$ is a {\it spherically symmetric manifold} or a {\it model manifold} if the Riemannian metric is given by
\begin{equation*}\label{e2}
ds^2 = d\rho^2+\psi^2(\rho)d\theta^2,
\end{equation*}
where $d\theta^2$ is the standard metric on $\mathbb S^{N-1}$, and $\psi\in \mathcal A$. In this case, we write $M\equiv M_\psi$; furthermore, we have $\sqrt A(\rho,\theta)=\psi^{N-1}(\rho) \, \eta(\theta) $ (for a suitable function $\eta$).

Note that for $\psi(r)=r$, $M=\mathbb R^N$, while for $\psi(r)=\sinh r$, $M$ is the $N$-dimensional hyperbolic space $\mathbb H^N$.

To most of our purposes, we shall assume that the following hypothesis is satisfied, where we denote by $\operatorname{Ric}_o(x)$ the {\it radial} Ricci curvature at $x$ w.r.t.~a given pole $ o \in M $ (see
Section \ref{RG} for some more detail):
\begin{equation}\tag{{\it H}}\label{H}
\begin{cases}
\textrm{(i)} & M \ \textrm{is a Cartan-Hadamard manifold of dimension $N\ge3$} \, ; \\
\textrm{(ii)} & \textrm{Ric}_o(x)\geq -C(1+\operatorname{dist}(x,o)^2) \ \textrm{for some }C\ge0 \, .
\end{cases}
\end{equation}

For instance, assumption \eqref{H} is satisfied if $M=\mathbb H^N$, and e.g.~on Riemannian models (see Section \ref{RG} below) associated with functions $\psi$ such that $ \psi^{\prime\prime}\ge 0 $ and $\psi(r)=e^{r^\alpha}$ for any $r>0$ large enough, for some $0 <\alpha\leq 2$.

Note that by \eqref{H} the Green function $G(x,y)>0$ on $M$ exists finite for all $x\neq y$ (see again Section \ref{RG}), i.e.~$M$ is {\it
nonparabolic}.

\medskip

Let $\mathcal M^+(M)$ be the set of positive Radon measures on $M$, with $\mathcal M_F^+(M):=\{\mu\in \mathcal M^+(M): \, \mu(M)<\infty \}$. We shall also denote by $ \mathcal M_F(M) $ the space of \emph{signed} finite measures on $M$, namely measures that can be written as the difference between two
elements of $\mathcal M_F^+(M)$.

\begin{den}\label{defsol1}
Given a measure $\mu\in \mathcal M_F(M)$, we say that a function
$u$ is a {\em weak solution} to problem \eqref{e64} if
\begin{equation}\label{e65}
u\in L^\infty((0,\infty);L^1(M)) \cap
L^\infty(M\times(\tau,\infty)) \quad \textrm{for all}\;\,
\tau>0\,,
\end{equation}
\begin{equation}\label{e66}
\nabla\!\left(u^m\right) \in L^2((\tau,\infty);L^2(M)) \quad \textrm{for all}\;\,
\tau>0\,,
\end{equation}
\begin{equation}\label{e67}
-\int_0^\infty\int_M u(x,t)\varphi_t(x,t) \, d\mathcal V(x) dt +
\int_0^\infty\int_M \langle \nabla (u^m)(x,t), \nabla
\varphi(x,t)\rangle \, d\mathcal V(x) dt \,=\,0
\end{equation}
for any $\varphi\in C^\infty_c(M\times (0,\infty))$, and
\begin{equation}\label{e68}
\lim_{t\to 0} \int_M u(x,t)\phi(x) \, d\mathcal
V(x)\,=\,\int_{M}\phi(x) \, d\mu(x)\,\quad \textrm{for any}\;\;
\phi\in C_b(M):=C(M)\cap L^\infty(M)\,.
\end{equation}
\end{den}
%
In fact we shall prove (see Proposition \ref{pro: cons} below) that weak solutions in the sense of Definition \ref{defsol1} are continuous curves in $ L^1(M) $.

\subsection{Existence and uniqueness results} Concerning existence of solutions starting from an initial finite (not necessarily positive)
Radon measure, that are allowed to change sign, we prove the next result. The strategy of the proof is similar to the one of \cite[Theorem
3.2]{GMPmu}, but however new ideas are necessary, since the method of proof of \cite[Theorem 3.2]{GMPmu} works only in the case of positive Radon
measures.

\begin{thm}\label{thmexi}
Let assumption \eqref{H}-\textrm{(i)} be satisfied. Let $\mu\in \mathcal M_F(M)$. Then there exists a weak solution $u$ to problem \eqref{e64}, which conserves the quantity
\begin{equation}\label{e69}
\mu(M)=\int_M u(x,t) \, d\mathcal V(x)\quad \textrm{for all}\;\,
t>0
\end{equation}
and satisfies the smoothing effect
\begin{equation}\label{e70}
\|u(t)\|_\infty \leq K t^{-\alpha}|\mu|(M)^{\beta} \quad \textrm{for
all}\,\, t>0\,,
\end{equation}
where $K$ is a positive constant which only depends on $m , N$ and
\begin{equation}\label{e71-pre}
\alpha:= \frac{N}{(m-1) N+2}\,,\quad \beta:=\frac{2}{(m-1) N+2}\,.
\end{equation}
\end{thm}

Note that this result can be extended, apart from the conservation of mass, to the case of the supercritical fast diffusion case $m\in((N-2)/N,1)$, see Remark \ref{supercritical} below.

\smallskip

Concerning uniqueness of nonnegative solutions, taking on the same initial positive finite measure, we show the following result. The ideas of the proof bear some similarities with the one given in \cite[Section 5]{GMPmu}, being based on the duality method of Pierre (see
\cite{Pierre}), but substantial differences occur, mainly due to the very different properties of the heat semigroup and the Green function on $M$,
related to our assumptions on sectional curvatures.

\begin{thm}\label{thmuni}
Let assumption \eqref{H} be satisfied. Let $u_1$ and $u_2$ be two nonnegative weak solutions to problem \eqref{e64}. Suppose that their initial
datum, in the sense of \eqref{e68}, is the same $\mu\in \mathcal M^+_F(M)$. Then $u_1=u_2$\,.
\end{thm}

Our final result concerns the existence and uniqueness of an initial trace for solutions to the differential equation in problem \eqref{e64}.

\begin{thm}\label{thmtracce}
Let assumption \eqref{H} be satisfied. Let $u$ be a weak solution of the differential equation in problem \eqref{e64}, in the sense that it satisfies \eqref{e65}--\eqref{e67}.
Then there exists $\mu \in \mathcal M_F(M)$ such that \eqref{e68} is satisfied for any $\phi\in C_c(M)$ or for $ \phi $ equal to a constant.

Under the additional assumption that $u\geq 0$, then the conclusion holds for any $\phi\in C_b(M)$, for some $\mu \in \mathcal M_F^+(M)$.
\end{thm}

\begin{oss}\rm
We point out that our existence and uniqueness results also hold in the linear case, i.e.~for $m=1$. To the best of our knowledge no results are available in the literature if the initial condition is a measure.
Note that for the heat equation the explosion rate $-\operatorname{dist}(x,o)^2$ for the Ricci curvature is a sharp condition for uniqueness as shown in \cite{IM}. For several other sharp results in the linear case see \cite{M1, I, I2, M2}
\end{oss}

\subsection{Superharmonic functions and modified mean-value properties}\label{Pot-MVP} In this section we establish a modified version of the mean-value inequality
for distributional superharmonic functions. It should be stressed that these results, although being of independent interest, will be essential in the
proofs of the potential theoretic results of Section \ref{Pot-meas}, which are in turn fundamental in the proof of uniqueness for solutions to
problem \eqref{e64}.

Unless otherwise stated, we assume here that $M$ is a nonparabolic manifold of dimension $N\ge2$, with $G$ being the minimal positive Green function of $M$.

Let $u:M \to (-\infty,+\infty]$ be a lower semicontinuous function (l.s.c.) function. For $r>0$ we define
\begin{equation}\label{e25}
\mathfrak{m}_r[u](x):= \int_{\left\{y\in M\,:\; G(x,y)=\frac 1 r\right\}} u(y) \big|\nabla_y G(x,y)\big|dS(y) \quad \textrm{for all}\;\, x\in M\,,
\end{equation}
where $dS$ is the $(N-1)$-dimensional Hausdorff measure on $M$\,. Moreover, for any $\alpha>0$, we set
\begin{equation}\label{e26}
\mathfrak{M}_r[u](x):= \frac{\alpha+1}{r^{\alpha+1}}\int_0^r \xi^\alpha \mathfrak{m}_\xi[u](x) \, d\xi \quad \textrm{for all}\;\, x\in M\,.
\end{equation}

Let us recall the well-known smooth coarea formula (see e.g.~\cite[Exercise III.12]{Chavel}). Let $\phi: M\to \mathbb R$ be of class $C^\infty(M)$ with $|\nabla \phi| \in L^\infty(M) $, and let $f:M \to \mathbb R$ be either nonnegative or in $L^1(M)$. Then
\[ \int_M f \, |\nabla \phi| \, d\mathcal V\,=\,\int_\mathbb R d\xi  \, \int_{ \left\{ y \in \phi^{-1}(\xi) \right\} } f\big(\phi^{-1}(\xi)\big) \, d S(y) \, . \]
By approximation it is not difficult to show that such formula is also true with the choices $\phi(y)=[G(x,y)]^{-1}$ and $  f(y) = u(y)[ G(x,y)]^{-\alpha} \chi_{\{ y \in M : \, G(x,y) > r \}} $, for each fixed $x\in M$. So, one can rewrite \eqref{e26} as
\begin{equation}\label{e26-bis}
\mathfrak{M}_r[u](x):= \frac{\alpha+1}{r^{\alpha+1}} \int_{\{y \in M: G(x,y)>\frac{1}{r}\}} u(y) [G(x,y)]^{-\alpha-2} \left| \nabla_y G(x,y) \right|^2 d\mathcal{V}(y)  \quad \textrm{for all}\;\, x\in M\,.
\end{equation}

\begin{den}
We say that a l.s.c.~function $u:M\to (-\infty,+\infty]$ is $\mathfrak{m}-$continuous if
\begin{equation*}\label{e27}
u(x)\,=\, \lim_{r\to 0} \mathfrak{m}_r[u](x) \quad \textrm{for all}\;\; x\in M\,.
\end{equation*}
Similarly, we say that $u$ is $\mathfrak{M}-$continuous if
\begin{equation*}\label{e28}
u(x)\,=\, \lim_{r\to 0} \mathfrak{M}_r[u](x) \quad \textrm{for all}\;\; x\in M\,.
\end{equation*}
\end{den}
We point out that if $u$ is continuous, then it is both $\mathfrak m-$continuous and $\mathfrak M-$continuous (see the proof of Lemma \ref{lemma9}).
Moreover, in general, if $u$ is $\mathfrak m-$continuous, it is also $\mathfrak M-$continuous.

\begin{den}\label{defsh}
We say that $u\in L^1_{\textrm{loc}}(M)$ is  {\em superharmonic} (resp. {\em subharmonic}) if
\begin{equation*}\label{e20}
\int_M u(x) \Delta \phi(x) \, d\mathcal V(x) \, \leq (\geq ) \, 0 \quad \textrm{for any}\;\; \phi\in C^\infty_c(M) \, , \ \phi\geq 0\,.
\end{equation*}
Moreover, $u\in L^1_{\textrm{loc}}(M)$ is {\em harmonic} if it is both {\em subharmonic} and {\em superharmonic}.
\end{den}

\begin{den}\label{defsh-2}
We say that a l.s.c.~function $u:M\to (-\infty,+\infty]$ is {\em $\mathfrak{m}-$superharmonic} if
\begin{equation*}\label{e29}
\mathfrak{m}_r[u](x) \leq u(x) \quad \textrm{for all}\;\, x\in M,\ \textrm{for a.e.}\ r>0\,.
\end{equation*}
Similarly, we say that $u$ is {\em $\mathfrak{M}-$superharmonic} if
\begin{equation*}\label{e30}
\mathfrak{M}_r[u](x) \leq u(x) \quad \textrm{for all}\;\, x\in M, \, r>0\,.
\end{equation*}

Furthermore, we say that $u$ is {\em $\mathfrak{m}-$subharmonic} if $-u$ is {\em $\mathfrak{m}$-super\-har\-monic}, while $u$ is {\em
$\mathfrak{M}-$sub\-har\-monic} if $-u$ {\em $\mathfrak{M}-$superharmonic}\,.

Finally, we say that $u$ is {\em $\mathfrak{m}-$harmonic} if it is both {\em $\mathfrak{m}-$subharmonic} and {\em $\mathfrak{m}-$superharmonic}\,,
while $u$ is {\em $\mathfrak{M}-$harmonic} if it is both {\em $\mathfrak{M}-$subharmonic} and {\em $\mathfrak{M}-$superharmonic}\,.
\end{den}

We have the following result, which will be proved in Section \ref{MVI-p}.
\begin{thm}\label{teo5}
$(i)$ Let $u$ be $\mathfrak{M}-$continuous, l.s.c.~and superharmonic. Then $u$ is $\mathfrak{M}-$superharmonic.

$(ii)$ Let $u$ be $\mathfrak{M}-$continuous, upper semicontinuous and subharmonic. Then $u$ is $\mathfrak{M}-$subharmonic.

\end{thm}

Of course, the above theorem implies that if $u$ is continuous and harmonic, then $u$ is $\mathfrak{M}-$harmonic\,, in agreement with the results of \cite{N}, which are given in principle for more regular functions.

\smallskip

We stress again that the classical mean-value formula (w.r.t.~the Riemannian measure of a ball) need not be valid, and that in principle only a mean-value inequality for nonnegative subharmonic functions holds (see \cite{LS}).

By means of minor modifications in the proof of Theorem \ref{teo5}, a local version of such results on general Riemannian manifolds (possibly
parabolic) can be obtained, without supposing that hypothesis \eqref{H} holds. In fact, we have the following.

\begin{cor}\label{cor11}
Let $\Omega\subset M$ be an open bounded subset. Let $u$ be $\mathfrak{M}-$continuous, l.s.c.~and superharmonic in $\Omega$. Then $u$ is
$\mathfrak{M}-$superharmonic in $\Omega$. Similar statements hold for subharmonic and harmonic functions.
\end{cor}
Note that in Corollary \ref{cor11}, the function $G$ in \eqref{e25} is meant to be replaced by the Green function of $-\Delta$ in $\Omega'$
completed with homogeneous Dirichlet boundary conditions at $\partial \Omega'$, where $\Omega'$ is any open bounded domain with smooth boundary such
that $\Omega\Subset\Omega'\,.$


We remark that, besides the previous ones, we expect that further results given in \cite{BonLan} can be extended to Riemannian manifolds. In
particular, it should be true that if a function $u$ is $\mathfrak{m}-$continuous, l.s.c.~and superharmonic, then it is $\mathfrak{m}-$superharmonic.
However, we limit ourselves to prove the results stated above, since they are the only ones we need in the study of existence and uniqueness for problem \eqref{e64}.

\subsection{A connection between the Green function and the porous medium equation}
In this section we state the nonlinear counterpart of a well-known result that relates the Green function to the heat kernel. In this case, the
role of the heat kernel is taken over by the fundamental solution $\mathcal B_{x_0}$ of problem \eqref{e64} with $\mu=\delta_{x_0}$, for each fixed
$x_0\in M$.

Suppose that hypothesis \eqref{H} is satisfied. Then by Theorem \ref{thmexi} the function $\mathcal B_{x_0}$ is well defined. If we drop such assumption the method developed in Section \ref{proof-ex} to construct $\mathcal B_{x_0}$ does not work. Nevertheless, the function $\mathcal B_{x_0}$ can always be defined as the monotone limit of approximate solutions to Dirichlet problems set in $B_R\times (0,\infty)$ (for the details, see the proof of Theorem \ref{thmBar} in Section \ref{alternative}). In general, we cannot in principle exclude that $\mathcal B_{x_0} = \infty$.

\begin{thm}\label{thmBar}
Let $M$ be a complete noncompact Riemannian manifold of dimension $ N \ge 2 $. For any $x_0\in M$, let $\mathcal{B}_{x_0}$ be the solution of problem \eqref{e64} with
$\mu=\delta_{x_0}$, meant in the sense described above. Then
\begin{equation}\label{e210}
G(x_0,y)\,=\, \int_0^\infty \mathcal{B}^m_{x_0}(y, t)\, dt\quad \textrm{for all}\ y\in M\,.
\end{equation}
In particular, the time integral in \eqref{e210} exists finite if and only if $M$ is nonparabolic.
\end{thm}

Note that, as a consequence of Theorem \ref{thmBar} and of symmetry of the Green function (see \eqref{e9} below), we have the identity
$$\int_0^\infty \mathcal B^m_{x_0}(y, t)\, dt\,=\,\int_0^\infty \mathcal B^m_{y}(x_0, t)\, dt \quad \textrm{for all}\ x_0, y\in M\,.$$

\begin{oss}\label{rem: barenblatt}{\em
Since sectional curvatures are by assumption nonpositive, Hessian comparison (see \eqref{e5}) shows that $B^E_{0}(\rho(x), t)$, where $B^E_0(|x|, t)$ is the Euclidean Barenblatt solution, is a supersolution of problem \eqref{e64} with $\mu=\delta_{0}$. By the comparison principle in bounded domains, it is not difficult to show that, as a consequence, if $u$ is a solution of \eqref{e64} with $\mu \equiv u_0$ and $\operatorname{supp}u_0$ compact, then $\operatorname{supp} u(t)$ is also compact for all $ t>0 $. For the details, we refer to the proof of Proposition \ref{pro: cons} in Section \ref{sect-cons}.

Moreover, in view of the construction of $\mathcal B_{0}$, by means of the same arguments as above, we have that $\mathcal B_{0}\leq B^E_{0}$ in $M\times (0, \infty)$. In particular, $\operatorname{supp} \mathcal B_{0}$ is compact.}
\end{oss}

\section{Preliminaries in Riemannian geometry and analysis on manifolds}\label{RG}
Let $M$ be a complete noncompact Riemannian manifold. Let $\Delta$
denote the standard Laplace-Beltrami operator, $\nabla$ the
gradient (with respect to the metric of $M$) and $d\mathcal V$ the
Riemannian volume element.

In \cite{Stric} it is shown that $-\Delta$, defined on
$C^\infty_c(M)$, is essentially self-adjoint in $L^2(M)$\,. In
particular, this implies that if $f\in L^2(M)$ with $\Delta f\in
L^2(M)$, then $\nabla f\in L^2(M)$, and there exists a sequence of
functions $\{f_j\}\subset C^\infty_c(M)$ such that
\begin{equation*}\label{e17}
f_j\to f\,,\quad \nabla f_j\to \nabla f \,,\quad \Delta f_j\to \Delta f\quad \textrm{in}\;\; L^2(M)\,.
\end{equation*}
In addition, for any $f, g\in L^2(M)$ with $\Delta f, \Delta g\in L^2(M)$ we have
\begin{equation*}\label{e18}
\int_M f\,\Delta g\, d\mathcal V \,=\,- \int_M \langle \nabla f, \nabla g \rangle\, d\mathcal V\,=\, \int_M g \Delta f\, d\mathcal V\,.
\end{equation*}

 It is direct to see that
the Laplace-Beltrami operator in the polar coordinates has the form
\begin{equation}\label{e1}
\Delta = \frac{\partial^2}{\partial \rho^2} + m(\rho, \theta)\frac{\partial}{\partial
\rho}+\Delta_{S_{\rho}},
\end{equation}
where $m(\rho, \theta):=\frac{\partial}{\partial \rho}\big(\log\sqrt{A}\big)$, $ A:=\det (A_{ij})$, $ \Delta_{S_\rho}$ is the Laplace-Beltrami operator on the submanifold $S_{\rho}:=\partial B(o, \rho)\setminus \operatorname{cut}(o)$ and $ B(o,\rho) $ denotes the Riemannian ball of radius $\rho$ centred at $o$ ($B(\rho)$ for short).
Furthermore, on model manifolds
\[\Delta = \frac{\partial^2}{\partial \rho^2}+ (N-1)\frac{\psi'}{\psi}\frac{\partial}{\partial\rho}+ \frac1{\psi^2}\Delta_{\mathbb S^{N-1}}\,,\]
where $\Delta_{\mathbb S^{N-1}}$ is the Laplace-Beltrami operator in $\mathbb S^{N-1}\,.$

\smallskip

Let us recall comparison results for sectional and Ricci
curvatures, which will be used in the sequel. For any $x\in
M\setminus\{o\}$, denote by $\textrm{Ric}_o(x)$ the
{\it Ricci curvature} at $x$ in the direction
$\frac{\partial}{\partial\rho}$. Let $\omega$ denote any pair of tangent
vectors from $T_xM$ having the form $\left(\frac{\partial}{\partial \rho}
,X\right)$, where $X$ is a unit vector orthogonal to
$\frac{\partial}{\partial\rho}$. Denote by $\textrm{K}_{\omega}(x)$ the {\it
sectional curvature} at the point $x\in M$ of the $2$-section
determined by $\omega$. By classical results (see e.g.~\cite{GW}, \cite[Section
15]{Grig}), if
\begin{equation}\label{e3a}
\textrm{K}_{\omega}(x)\leq -\frac{\widetilde \psi''(\rho)}{\widetilde \psi(\rho)}\quad \textrm{for all}\;\; x\equiv(\rho,\theta)\in M\setminus\{o\},
\end{equation}
for some function $\widetilde \psi\in \mathcal A$, then
\begin{equation*}\label{e3}
m(\rho, \theta)\geq (N-1)\frac{\widetilde \psi'(\rho)}{\widetilde \psi(\rho)}\quad \textrm{for all}\;\; \rho>0, \, \theta \in \mathbb S^{N-1}\,.
\end{equation*}
Moreover, (see e.g.~\cite[Section 3]{Grig})
\begin{equation}\label{e302}
\mathcal V(B_R)\geq \omega_N \int_0^R \widetilde \psi^{N-1}(\rho) \, d\rho\,,
\end{equation}
where $\omega_N$ is the measure of the unit sphere $\mathbb S^{N-1}\,.$

On the other hand, if
\[\textrm{Ric}_{o}(x)\geq
-(N-1)\frac{\psi''(\rho)}{\psi(\rho)}\quad \textrm{for all}\;\; x\equiv(\rho,\theta)\in M\setminus\{o\},\] for some function $\psi\in \mathcal A$, then
\begin{equation}\label{e4}
m(\rho, \theta)\leq (N-1)\frac{\psi'(\rho)}{ \psi(\rho)}\quad \textrm{for all}\;\; \rho>0, \, \theta \in \mathbb S^{N-1}\,.
\end{equation}

Note that if $M_\psi$ is a model manifold, then for any $x\equiv(\rho, \theta)\in M_\psi\setminus\{o\}$ we have
\[\textrm{K}_{\omega}(x)=-\frac{\psi''(\rho)}{\psi(\rho)} \]
and
\[\textrm{Ric}_{o}(x)=-(N-1)\frac{\psi''(\rho)}{\psi(\rho)}\,.\]

Since in view of hypothesis \eqref{H} we have $\textrm{K}_{\omega}(x)\leq 0$, we can infer that condition \eqref{e3a} is trivially satisfied with
$\widetilde \psi(\rho)=\rho$. Therefore,
\begin{equation}\label{e5}
m(\rho, \theta)\geq \frac{N-1}{\rho}\quad \textrm{for all}\;\; x\equiv(\rho, \theta)\in M\setminus \{o\}\,.
\end{equation}

\medskip

Let $\textrm{spec}(-\Delta)$ be the spectrum in $L^2(M)$ of the operator
$-\Delta$. Note that (see \cite[Section 10]{Grig})
\[\textrm{spec}(-\Delta)\subseteq [0,\infty)\,.\]
As a consequence of \eqref{H}-(i), the {\it Sobolev inequality}
\begin{equation}\label{e12}
\|f\|_{\frac{2 N}{N-2}}\leq C_S \| \nabla f\|_2 \quad \textrm{for all}\;\; f\in C^\infty_c(M)
\end{equation}
holds for some positive constant $C_S>0$, which is equivalent to  the {\it Faber-Krahn} inequality
\begin{equation}\label{e13}
\lambda_1(\Omega)\geq C_{FK} [\mathcal V(\Omega)]^{-\frac 2 {N}}
\end{equation}
for some positive constant $C_{FK}$, for any bounded regular domain $\Omega\subset M$. Here $\lambda_1(\Omega)$ denotes the first eigenvalue for the operator $-\Delta$ in $L^2(\Omega)$, completed with homogeneous Dirichlet boundary conditions on $\partial \Omega$. Moreover, for some positive constant $C_{N}$ one has
\begin{equation}\label{e14}
\mathcal V (B_R(x))\geq C_{N}  R^{N}\quad \textrm{for any}\,\; x\in M, R>0\,\,.
\end{equation}
Inequalities \eqref{e12} and \eqref{e13} and their connection are classical results, which follow e.g.~from \cite[Exercise 14.5, Corollary 14.23,
Remark 14.24]{Grig3} or \cite[Lemma 8.1, Theorem 8.3]{Hebey}. Furthermore, \eqref{e14} is due to \eqref{H}-(i) and \eqref{e302} with
$\widetilde{\psi}(\rho)=\rho.$

Let $G(x,y)$ be the {\it Green function} on $M$. Note that a priori (see \cite{Grig}) either $G(x,y)=\infty$ for all $x,y\in M$ or $G(x,y)<\infty$ for all $x\neq y.$

Since $M$ is by assumption a Cartan-Hadamard manifold and hence sectional curvatures are nonpositive, standard Hessian comparisons imply that
\begin{equation}\label{e16}
G(x,y)\leq  \widetilde C \, [\operatorname{dist}(x,y)]^{2- N}\quad \textrm{for all}\,\, x,y\in M\,,
\end{equation}
for a suitable $\widetilde C>0$ (we refer e.g.~to \cite[Theorem 4.2]{Grig2} and \eqref{e35} below). In particular, the {\it Green} function $G(x,y)$ is finite for any $x\neq y$ and vanishes as $ \operatorname{dist}(x,y) \to \infty $.
Furthermore (see \cite[Section 4]{Grig}),
\begin{equation}\label{e7}
G(x,y) \sim \widetilde C\,[\operatorname{dist}(x,y)]^{2-N} \quad \textrm{as}\ \operatorname{dist}(x,y)\to 0 \ \textrm{(for any fixed $y$)} \, ,
\end{equation}
\begin{equation}\label{e8}
G(x,y)>0\quad \textrm{for all}\;\; x,y\in M\,,
\end{equation}
\begin{equation}\label{e9}
G(x,y)=G(y,x) \quad \textrm{for all}\;\; x,y\in M\,.
\end{equation}
In addition,
\begin{equation}\label{e57}
\textrm{for each fixed} \;\, y\in M, \ x\mapsto G(x,y)\;\,
\textrm{is of class} \,\; C^\infty(M\setminus\{y\}) \, ,
\end{equation}
\begin{equation*}\label{e10}
\Delta_x G(x, y) = 0 \quad \textrm{for any}\ x\in M\setminus \{y\}\,,
\end{equation*}
and
\begin{equation}\label{e11}
\int_M  G(x,y) \Delta\phi(x) \, d\mathcal V(x)\,=\,-\phi(y)  \,\leq \,0
\end{equation}
for any $\phi\in C^\infty_c(M)$ with $ \phi\geq 0$. Moreover, by Sard's theorem,
for all $x\in M$ and a.e.~(possibly depending on $x$) $a>0$, one has $\nabla_y G(x,y)\not=0$  on the level set $\{y\in M\,:\,
G(x,y)=a\}$. In particular such level sets are smooth.

\medskip

Let $h$ be the {\it heat kernel} on $M$; we have the identity
\begin{equation}\label{e35}
G(x,y)=\int_0^\infty h(x,y,t) \, dt \quad \textrm{for all } x,y \in M
\end{equation}
(see \cite{Grig}). Moreover, let $\{T_t\}_{t\geq 0}$ denote the heat semigroup on $M$. The minimal positive solution of the Cauchy
problem for heat equation
\[
\begin{cases}
u_t = \Delta u & \textrm{in } M\times (0,\infty) \, , \\
u = u_0 \in L^1(M) \, , \, u_0\geq 0 & \textrm{on } M \times \{0\} \, ,
\end{cases}
\]
can be written as
\[T_t[u_0](x)=\int_M h(x,y,t) u_0(y) \, d\mathcal{V}(y) \quad \textrm{for all } x\in M, \, t\geq 0 \,.\]
Note that
\begin{equation}\label{e33}
\|T_t \phi \|_p \leq \|\phi \|_p \quad \textrm{for all } t>0 \, , \
p \in [1,\infty] \, , \ \phi\in L^p(M)\,.
\end{equation}
Furthermore, as a consequence of \eqref{e12}, we have
\begin{equation}\label{e31}
\|T_t \phi\|_\infty \leq \frac C{t^{{N}/2}} \|\phi\|_1\quad \textrm{for any } t>0, \, \phi\in L^1(M)\,,
\end{equation}
for some $C=C(N)>0$ (see e.g.~\cite[Chapter 4]{Dav}).

\section{Auxiliary results in potential analysis on Riemannian manifolds}\label{Pot}
This section is devoted to establishing some crucial results for superharmonic functions and potentials of Radon measures, the latter being closely related to the former. Here $M$ will always be assumed, unless otherwise stated, to be a nonparabolic Cartan-Hadamard
manifold of dimension $N\ge2$.
\subsection{Proof of the modified mean-value inequality and properties of superharmonic functions}\label{MVI-p}
In order to show the modified mean-value inequality, we need a preliminary lemma.
\begin{lem}\label{lemma9}
For each fixed $y\in M$, the function $x\mapsto G(x,y)$ from $M$ to $[0,+\infty]$ is superharmonic. Moreover, it is both $\mathfrak{m}-$ and
$\mathfrak{M}-$continuous.
\end{lem}
\begin{proof}
In view of \eqref{e7} and \eqref{e11}, the function $x\mapsto G(x,y)$ is superharmonic. Furthermore, an easy application of the divergence theorem yields, for any $x\in M$, for a.e.\ $r>0$,
\begin{equation}\label{e50}
-\int_{\{y\in M\,:\, G(x,y)>\frac 1 r\}} G(x,y)\Delta \phi(y) d\mathcal V(y) = -\mathfrak m_r[\phi](x) -\frac 1 r \int_{\{y\in M\,:\, G(x,y)>\frac
1 r\}} \Delta \phi(y) d\mathcal V(y) + \lim_{\rho\to 0} \mathfrak m_\rho[\phi](x)\,
\end{equation}
for any $\phi\in C^2(M)$. This can be shown exactly as in formula (11.4) in \cite{BonLan}, upon noting that $\lim_{r\to 0}\mathfrak m_{\rho}[\phi](x)$ exists, as proved in formula (11.2) and just above (11.7) in \cite{BonLan}.

Now, we choose $\phi=\xi$ with $\xi\in C^\infty_c(M)$, $ \xi=1$ in a neighbourhood of $x$, and $r>0$ so large that $\operatorname{supp}\, \xi\subset
\{y\in M\,:\, G(x,y)>\frac 1 r\}$. Hence, using \eqref{e11}, \eqref{e50}, an integration by parts, and the fact that $\mathfrak m_r[\phi](x)=0$, we obtain
\begin{equation}\label{e51}
\lim_{\rho\to 0}\int_{\{y\in M\,:\, G(x,y)=\frac 1 {\rho}\}} |\nabla_y G(x,y)| dS(y)\,=1\,.
\end{equation}
From \eqref{e51} it easily follows that any continuous function on $M$ is automatically $\mathfrak{m}-$, and so $\mathfrak{M}-$con\-tinuous.
Therefore, for each $y\in M$, the function $x\mapsto G(x,y)$ is $\mathfrak m-$continuous at any $x\in M\setminus \{y\}$. We are left with showing
that it is $\mathfrak m-$continuous also at $x=y$. This is a straightforward consequence of the very definition of $\mathfrak m_r$ and \eqref{e51}:
\[
\begin{aligned}
\lim_{r\to 0}\mathfrak m_r[G(\cdot, y)](y) = & \lim_{r\to 0}\int_{\left\{z\in M\,:\; G(y,z)=\frac 1 r\right\}} G(y,z) \big|\nabla_z G(y,z)\big|dS(z) \\
= &\lim_{r\to 0}
\frac 1 r \int_{\left\{z\in M\,:\; G(y,z)=\frac 1 r\right\}}
 \big|\nabla_z G(y,z)\big|dS(z)=\infty\,.
\end{aligned}
 \]
Hence the function $x\mapsto G(x,y)$ is $\mathfrak M-$continuous, too. This completes the proof.
\end{proof}

\begin{proof}[Proof of Theorem \ref{teo5}]
We shall prove that for every $x\in M$ the function $r\mapsto \mathfrak M_r[u](x)$ is nonincreasing in $(0,\infty)$.  Note that this property
combined with the fact that $u$ is $\mathfrak{M}-$continuous easily gives the thesis.

Now, let $\psi\in C^\infty([0,\infty))$ with $\psi\geq 0$, $ \psi$
constant in $[0,\varepsilon),\, \psi=0$ in $[R, \infty)$ for some
$R>\varepsilon>0\,.$ Fix any $x_0\in M$. Define
\begin{equation}\label{e62}
\phi(x):= \psi\left(\frac{1}{G(x_0, x)}\right) \quad \textrm{for all} \;\;  x \in M \, ,
\end{equation} with the obvious convention that
$\phi(x_0)=\psi(0)\,.$ In view of \eqref{e57} and \eqref{e16}, we
have that $\phi\in C^\infty_c(M)\,.$ Since $u$ is superharmonic,
due to Definition \ref{defsh} there holds
\begin{equation}\label{e58}
\int_M u \Delta \phi \, d\mathcal V \leq 0\,.
\end{equation}
A straightforward computation yields
\begin{equation}\label{e59}
\Delta \phi(x)\,=\,\frac{|\nabla_x G(x_0, x)|^2}{[G(x_0,
x)]^4}\left[\psi''\left(\frac 1{G(x_0, x)}\right) + 2 G(x_0, x)
\psi'\left(\frac1{G(x_0, x)}\right)\right] \quad \textrm{for
all}\;\; x\in M\,.
\end{equation}
In view of \eqref{e8}, of the explicit form of $\Delta \phi(x)$ given above, and of the discussion after formula \eqref{e11},  we can apply the smooth coarea
formula (see again \cite[Exercise III.12]{Chavel}), \eqref{e58},
\eqref{e59} to get
\begin{equation}\label{e60}
\begin{aligned}
0&\geq \int_M u \Delta \phi \, d\mathcal V \\
& =\int_0^\infty \int_{\{x\in M\,:\,\frac 1{G(x_0, x)}= t\}} u(x)
\frac{|\nabla_x G(x_0, x)|}{[G(x_0, x)]^2} \left[\psi''\left(\frac
1{G(x_0, x)}\right) + 2 G(x_0, x) \psi'\left(\frac1{G(x_0,
x)}\right)\right] dt \\
& = \int_0^\infty t^2\left[\psi''(t) +
\frac{2\psi'(t)}{t}\right]\int_{\{x\in M\,:\,\frac 1{G(x_0, x)}=
t\}} u(x)\big|\nabla_x G(x_0, x)\big| dS(x) dt \\
& = \int_0^\infty[t^2\psi''(t) + 2t \psi'(t)]\mathfrak m_t[u](x_0) \,
dt =\int_0^\infty(t^2 \psi'(t))'\mathfrak m_t[u](x_0) \, dt\,.
\end{aligned}
\end{equation}
Given any $\eta\in C^\infty_c((0, \infty))$ with $\eta\geq 0$, we
can pick
\[\psi(t):= \int_t^\infty \frac{\eta(s)}{s^2} \, ds \quad \textrm{for all } t\in [0, \infty)\,.\]
Using such $\psi$ in \eqref{e62} and \eqref{e60} we obtain
\begin{equation}\label{e63}
\int_0^\infty \eta'(t) \mathfrak m_t[u](x_0) \, dt \geq 0\quad
\textrm{for all } \eta\in C^\infty_c((0, \infty)), \, \eta\geq
0\,.
\end{equation}
By \cite[Lemma 8.2.13]{BLU}, \eqref{e63} implies that the function
$r\mapsto \mathfrak M_r[u](x)$ is nonincreasing in $(0,\infty)$.
This completes the proof.
\end{proof}

As a consequence of Lemma \ref{lemma9} and Theorem \ref{teo5} we
obtain the next result.
\begin{cor}\label{cor10}
For each $y\in M,$ the function $x\mapsto G(x,y)$ is
$\mathfrak{M}-$superharmonic.
\end{cor}

We have two further lemmas, concerning superharmonic functions, which will be used in the sequel.
\begin{lem}\label{lemma7}
Let $u$ be an $\mathfrak{M}-$superharmonic function. Then $u$ is
$\mathfrak{M}-$continuous.
\end{lem}
\begin{proof}
Let $x\in M$. As a consequence of Definition \ref{defsh-2} we immediately deduce that
\begin{equation}\label{e51a}
u(x) \geq \limsup_{r\to 0}\mathfrak M_r[u](x)\,.
\end{equation}
Now, let $\varepsilon>0$ and $ u(x)<+\infty $ (the proof on the case $ u(x)=+\infty $ is analogous). Since $u$ is l.s.c.~at $x$, there exists $\widetilde r_\varepsilon>0$ such that
\begin{equation}\label{e52}
\inf_{B_{\widetilde r_\varepsilon}(x)} u \geq u(x) - \varepsilon\,.
\end{equation}
Due to \eqref{e16}, there exists $\bar r_\varepsilon>0$ such that
\begin{equation}\label{e53}
\left\{y\in M:\, G(x,y)=\frac 1{\rho}\right\}\subset  B_{\widetilde r_\varepsilon}(x)\quad \textrm{for all}\,\, 0< \rho\leq \bar r_\varepsilon\,.
\end{equation}
Hence, in view of \eqref{e52} and \eqref{e53}, we obtain
\begin{equation}\label{e54}
\mathfrak M_r[u](x)\geq
[u(x)-\varepsilon] \, \frac{\alpha+1}{r^{\alpha+1}} \int_0^r \rho^\alpha
\int_{\{y\in M\,:\, G(x,y)=\frac 1{\rho}\}}\big|\nabla_y G(x,y)\big|dS(y)
d\rho \quad \textrm{for all}\;\, 0<r\leq \bar r_\varepsilon\,.
\end{equation}
Due to \eqref{e51}, letting $r\to 0$ in \eqref{e54} yields
\begin{equation}\label{e55}
\liminf_{r\to 0} \mathfrak M_r[u](x) \geq u(x)-\varepsilon\,.
\end{equation}
The conclusion follows from \eqref{e51a} and \eqref{e55}, since
$\varepsilon$ is arbitrary.
\end{proof}

\begin{lem}\label{lemma8}
Let $\{u_n\}$ be a sequence of $\mathfrak M-$superharmonic
functions\,. Then the function $$x\mapsto \liminf_{n\to \infty}
u_n(x)$$ is $\mathfrak M-$superharmonic\,.
\end{lem}
\begin{proof}
Since for each $n\in \mathbb N$, $u_n$ is $\mathfrak
M-$superhamonic, it satisfies
\begin{equation}\label{e56}
u_n(x) \geq \frac{\alpha+1}{r^{\alpha+1}}\int_0^r \rho^\alpha
\int_{\{y\in M\,:\,G(x,y)=\frac 1{\rho}\}} u_n(y)\big|\nabla_y G(x, y)\big|
dS(y) d\rho\,\quad \textrm{for all}\;\; x\in M\,.
\end{equation}
By Fatou's Lemma applied to the right-hand side of \eqref{e56},
there holds
\[\liminf_{n\to \infty} u_n(x) \geq \frac{\alpha+1}{r^{\alpha+1}}
\int_0^r \rho^\alpha \int_{\{y\in M\,:\, G(x,y)=\frac
1{\rho}\}}\liminf_{n\to \infty} u_n(y)\big|\nabla_y G(x,y)\big| dS(y)
d\rho\,,
\]
namely $\liminf_{n\to \infty} u_n$ is $\mathfrak
M-$superharmonic.
\end{proof}

\subsection{Potentials of Radon measures and their properties}\label{Pot-meas}
We start by recalling the definition of vague convergence for sequence of Radon measures.

\begin{den}\label{def1}
Given a sequence $\{\mu_n\}\subset \mathcal M^+(M)$ and $\mu\in \mathcal M^+(M)$, we say that $\mu_n$ {\em converges vaguely} to $\mu$, and we write
\[\mu_n \rightharpoonup \mu \quad \textrm{as}\ n\to \infty \,,\]
if
\begin{equation}\label{e20z}
\int_M \phi \, d \mu_n  \to  \int_M \phi \, d\mu \quad \textrm{as } n\to \infty \, \quad \textrm{for all}\ \phi\in C_c(M)\,.
\end{equation}
The same definition holds for a sequence $\{\mu_n\} \subset \mathcal M_F(M)$ and $\mu\in \mathcal M_F(M)$. In such case the validity of \eqref{e20z}
plus the condition $ \sup_{n} |\mu_n|(M) < \infty $ is equivalent to the validity of \eqref{e20z} for all $ \phi \in C_0(M):=\{ \phi \in C(M): \,
\phi(x)\to 0 \ \textrm{as} \ d(x,o) \to \infty \} $, see e.g.~\cite[Definition 1.58]{AFP}.

A well-known compactness result asserts that if $ \sup_{n} |\mu_n|(M) < \infty $ then there exists $ \mu \in \mathcal{M}_F(M) $ such that \eqref{e20z} holds for all $ \phi \in C_0(M) $ along a subsequence \cite[Theorem 1.59]{AFP}.

Furthermore, the vague convergence implies a lower semicontinuity property:
$$ |\mu|(M) \le \liminf_{n \to \infty} |\mu_n|(M) \, . $$
\end{den}

For any $\mu\in \mathcal M^+(M)$ we define its {\em potential} as
\begin{equation*}\label{e21}
\mathcal G^{\mu}(x):= \int_M G(x,y) \, d\mu(y) \quad \textrm{for all}\ x \in M \,.
\end{equation*}
Note that, in general, $\mathcal G^\mu$ is a function from $M$ to $[0,+\infty]$. When $ d\mu(y)=f(y)d\mathcal{V}(y) $ for some measurable function $f
\ge 0$, we shall use the simplified notation
\begin{equation}\label{e21-bis}
\mathcal G^{f}(x):= \int_M G(x,y) f(y) \, d\mathcal{V}(y) \quad \textrm{for all}\ x \in M \, .
\end{equation}
The same definition holds for any $\mu\in \mathcal M_F(M)$, namely $ \mathcal G^{\mu} = \mathcal G^{\mu_+} - \mathcal G^{\mu_-} $. In this case $ \mathcal G^{\mu}(x) $ only makes sense for almost every $ x \in M $: by means of Tonelli's theorem and estimate \eqref{e16}, it is straightforward to show that potentials of finite Radon measures are at least $ L^1_{\rm loc}(M) $ functions.

\smallskip

The main goal of this section is to prove the next result.
\begin{pro}\label{teo10}
Let $\{\mu_n\}\subset \mathcal M^+(M)$ and $\mu\in \mathcal M^+(M)$, with $\mu_n \rightharpoonup \mu$. Suppose that for each compact subset $K\subset
M$ and for any $\varepsilon>0$ there exists $R_\varepsilon>0$  such that
\begin{equation}\label{e23}
\int_{B^c_R} \int_K G(x,y)\, d\mathcal V(y) d\mu_n(x) \le \varepsilon \quad \textrm{for any}\ R>R_\varepsilon, \, n \in \mathbb N\,.
\end{equation}
Then
\begin{equation}\label{e24a}
\mathcal G^{\mu}(x) \,=\, \liminf_{n\to \infty} \mathcal G^{\mu_n}(x)\quad \textrm{for \emph{every}}\ x\in M\,.
\end{equation}
\end{pro}
We point out that Proposition \ref{teo10} will have a key role in the proof of Theorem \ref{thmuni}. In particular, the fact that \eqref{e24a} holds {\it for every} $x\in M$ will be fundamental.

\medskip

The proof of Proposition \ref{teo10} requires some preliminary tools.

\begin{pro}[Principle of descent]\label{prop1}
Let $\{\mu_n\}\subset \mathcal M^+(M)$ and $\mu\in \mathcal M^+(M)$. Suppose that $\mu_n \rightharpoonup \mu$. Then
\begin{equation}\label{e22}
\mathcal G^{\mu}(x)\leq \liminf_{n\to \infty} \mathcal G^{\mu_n}(x)\quad \textrm{for all}\,\, x\in M\,.
\end{equation}
\end{pro}
\begin{proof}
Assume first that there exists a compact subset $K$ such that
$\operatorname{supp}\mu_n\subset K$ for any $n\in \mathbb N$ and
$\operatorname{supp}\mu\subset K\,.$ For each $\varepsilon>0$
define
\[ G_\varepsilon(x,y):= \phi_\varepsilon\left(\frac 1{G(x,y)}\right)\quad \textrm{for all}\ x,y\in M \, , \]
where
\[\phi_\varepsilon(r):=\begin{cases}\frac 1{\varepsilon} & r\leq
\varepsilon \, , \\
\frac 1 r & r>\varepsilon\,.
\end{cases} \]
Note that $G_\varepsilon$ is continuous and bounded in $M\times M$;
furthermore, for each $\varepsilon>0$,
\begin{equation}\label{e39}
G_\varepsilon(x,y) \leq G(x,y)\quad \textrm{for all}\ x,y\in
M\,,
\end{equation}
and
\begin{equation}\label{e40}
G_\varepsilon(x,y) \to G(x,y) \quad \textrm{as}\ \varepsilon\to
0 \quad \textrm{for all}\ x,y\in M\,.
\end{equation}
Hence, in view of \eqref{e39} and of the fact that $\mu_n
\rightharpoonup \mu$,
\begin{equation}\label{e41}
\int_M G_\varepsilon(x,y) \, d\mu(y) = \lim_{n\to \infty}\int_M
G_\varepsilon(x,y) \, d\mu_n(y) \leq \liminf_{n\to \infty} \int_M
G(x,y) \, d\mu_n(y)\quad \textrm{for all}\ x\in M \, .
\end{equation}
As a consequence of \eqref{e40}, \eqref{e41}, and Fatou's
Lemma, we obtain
\begin{equation*}\label{e42}
\mathcal G^\mu(x) = \int_M \lim_{\varepsilon\to 0}
G_\varepsilon(x,y) \, d\mu(y)\leq \liminf_{\varepsilon\to 0}\int_M
G_\varepsilon(x,y) \, d\mu(y)\leq \liminf_{n\to \infty} \int_M
G(x,y) \, d\mu_n(y) \, .
\end{equation*}
for all $x\in M$.

\smallskip
In order to complete the proof, we have to get rid of the assumption
$\operatorname{supp}\mu_n\subset K$ for any $n\in \mathbb N$ and
$\operatorname{supp}\mu\subset K$. To this end, note that since
$\mu$ is locally finite, the function $R\mapsto
\mu(B_R)$ is locally bounded and nondecreasing, thus its jump set
is countable. Therefore, we can select an increasing sequence
$\{R_k\}\subset (0,\infty)$ such that $\mu(\partial B_k)=0$. This
implies that $\mu_n^k:=\mu_n\rfloor B_k \rightharpoonup \mu\rfloor
B_k=: \mu^k$ as $n\to \infty$, for each $k\in \mathbb N$ (see \cite[Proposition 1.62]{AFP}). So,
\[\mathcal G^{\mu^k}(x)\leq \liminf_{n\to \infty} \mathcal G^{\mu_n^k}(x)\leq  \liminf_{n\to \infty} \mathcal G^{\mu_n}(x) \quad \textrm{for all}\ x\in M\,.
\]
Hence \eqref{e22} follows by letting $k\to \infty$ in the above inequality, in view of the monotone convergence theorem.
\end{proof}

\begin{lem}\label{pro2}
Let $\mu \in \mathcal M^+(M)$. Then $\mathcal G^{\mu}: M \mapsto [0, +\infty]$ is a l.s.c.~function.
\end{lem}
\begin{proof}
Given $x_0\in M$, take any sequence $\{x_n\}\subset M$ with
$x_n\to x_0$. Due to Fatou's Lemma, the continuity of $y\mapsto
G(x_0,y)$ in $M\setminus\{x_0\}$ for each $x_0\in M$ and
\eqref{e7}, we get
\[
\mathcal G^\mu(x_0)=\int_M \lim_{n\to \infty} G(x_n,y) \, d\mu(y) \leq
\liminf_{n\to \infty} \int_M  G(x_n, y) \, d\mu(y) =\liminf_{n\to \infty} \mathcal G^\mu(x_n) \, .
\]
This completes the proof.
\end{proof}

\begin{lem}\label{pro3}
Let the assumptions of Proposition \ref{teo10} be satisfied. Then
\begin{equation}\label{e24}
\mathcal G^{\mu}(x) \,=\, \liminf_{n\to \infty} \mathcal G^{\mu_n}(x)\quad \textrm{for}\;\,  \mathcal{V}-  \textrm{a.e.}\ x\in M\,.
\end{equation}
\end{lem}
\begin{proof}
We shall proceed by contradiction. In fact, suppose that for the
set
\[ E:= \left\{x\in M: \, \mathcal G^\mu(x) < \liminf_{n\to \infty} \mathcal G^{\mu_n}(x)\right\}\,\]
we have $\mathcal V(E)>0$. We can therefore select a compact
subset $K\subset E$ with $\mathcal V(K)>0$. By Fatou's Lemma
and the very definition of $E$, we have
\begin{equation}\label{e43}
\int_K \mathcal G^\mu \, d\mathcal V < \int_K \liminf_{n\to \infty} \mathcal G^{\mu_n} \,
d\mathcal V \leq \liminf_{n\to \infty}\int_K \mathcal G^{\mu_n} \, d\mathcal
V\,.
\end{equation}
Note that for any $\nu \in \mathcal M^+(M)$, by Tonelli's theorem there holds
\begin{equation*}\label{e44}
\int_K \mathcal G^{\nu} \, d\mathcal V = \int_M \phi_K \, d\nu \,,
\end{equation*}
where
$$
\phi_K(x):=\int_K G(x,y) \, d\mathcal V(y)\quad \textrm{for all}\ x\in M \, .
$$
Since $\phi_K=\mathcal G^{\chi_K}$, Lemma \ref{Lemma0} below implies $\phi_K\in C(M)\cap
L^\infty(M)$.  For any $R>0$ let $\phi_K^R$ be a continuous
function on $M$ with
\[ \phi_K^R(x)=
\begin{cases}
\phi_K(x) & \textrm{for any}\ x\in B_R \, , \\
0 & \textrm{for any}\ x\in B^c_{R+1} \, ,
\end{cases}
\]
and
\[
\phi_K^R \leq \phi_K \quad \textrm{in}\ M\, .
\]
We have:
\begin{equation}\label{e45}
\left|\int_M \phi_K \, d\mu - \int_M \phi_K \, d\mu_n \right|\leq
\underbrace{\int_M (\phi_K - \phi_K^R) \, d\mu_n}_{I_1}
+\underbrace{\int_M (\phi_K -\phi_K^R) \, d\mu}_{I_2} +
\underbrace{\left| \int_M \phi_K^R \, d\mu_n -\int_M \phi_K^R \,
d\mu\right|}_{I_3}\,.
\end{equation}
Thanks to \eqref{e23}, $I_1$ can be estimated as follows:
for any $\varepsilon>0$ there exists $R_\varepsilon>0$ such that for
all $R>R_\varepsilon$, $ n\in \mathbb N$, there holds
\begin{equation}\label{e46}
0\leq I_1 \leq \int_{B_R^c} \phi_K \, d\mu_n \, \leq
\,\varepsilon\,.
\end{equation}
Now, for any $R_2>R_1>1$ let $\xi\in C(M)$ with
\[ \xi_K^{R_1, R_2}(x) = \xi(x)=
\begin{cases}
\phi_K(x) & \textrm{for any}\ x\in B_{R_2}\setminus B_{R_1} \, , \\
0 & \textrm{for any}\ x\in B^c_{R_2+1}\cup B_{R_1-1}\,,
\end{cases}
\]
and
\[
\xi \leq \phi_K \quad \textrm{in}\ M \, .
\]
Now we observe that, since $\mu_n\rightharpoonup \mu$ as $n\to
\infty$, property \eqref{e23} and Fatou's Lemma imply
\[
\int_{B_{R_1}^c}\phi_K \, d\mu \leq \liminf_{R_2\to \infty} \int_M
\xi \, d\mu =  \liminf_{R_2\to \infty}\lim_{n\to \infty} \int_M \xi \,
d\mu_n \leq \limsup_{n\to \infty} \int_{B_{R_1-1}^c} \phi_K \, d\mu_n \leq
\varepsilon
\]
provided $R_1>R_{\varepsilon} + 1 $. This yields
\begin{equation}\label{e47}
0\leq I_2\leq \varepsilon \quad \textrm{for all}\ R>R_{\varepsilon}+1\,.
\end{equation}
Moreover, $I_3\to 0$ as $n\to \infty$ as a consequence of the very
definition of vague convergence. Hence, letting $n\to \infty$ in
\eqref{e45}, choosing $R>R_{\varepsilon} + 1$, using
\eqref{e46} and \eqref{e47}, we deduce
\begin{equation}\label{e48}
\limsup_{n\to \infty} \left|\int_M \phi_K \, d\mu - \int_M \phi_K \, d\mu_n \right|\leq 2 \varepsilon \, .
\end{equation}
It is apparent that \eqref{e48} is in contradiction with \eqref{e43}. Thus, \eqref{e24} follows.
\end{proof}

\begin{lem}\label{lemma6}
Let $\mu\in \mathcal M^+(M)$. Then $\mathcal G^\mu$ is $\mathfrak M-$superharmonic.
\end{lem}
\begin{proof}
Let $x\in M$ and $r>0$. Thanks to Tonelli's theorem and Corollary \ref{cor10}, we have
\[
\begin{aligned}
\mathcal G^\mu(x)&=\int_M G(x,y) \, d\mu(y) \geq \int_M \mathfrak M_r[G(\cdot,
y)](x) \, d\mu(y) = \frac{\alpha+1}{r^{\alpha+1}}\int_M \int_0^r
\rho^\alpha \mathfrak m_\rho[G(\cdot, y)](x) \, d\rho d\mu(y)\\
& =\frac{\alpha+1}{r^{\alpha+1}}\int_M \int_0^r \rho^\alpha
\int_{\{z\in M\,:\,G(x,z)= 1/{\rho}\}} G(y,z)\big|\nabla_z G(x,z)\big| \,
dS(z) d\rho d\mu(y)\\ &=\frac{\alpha+1}{r^{\alpha+1}}\int_0^r
\rho^\alpha \int_M \int_{\{z\in M\,:\,G(x,z)= 1/{\rho}\}}
G(y,z)\big|\nabla_z G(x,z)\big| \, dS(z) d\mu(y) d\rho \\
& = \frac{\alpha+1}{r^{\alpha+1}} \int_0^r \rho^\alpha
\int_{\{z\in M\,:\, G(x,z)= 1/{\rho}\}}\overbrace{\int_M
G(z,y) \, d\mu(y)}^{\mathcal G^{\mu}(z)} \big|\nabla_z G(x,z)\big| \, dS(z) d\rho =
\mathfrak M_r[\mathcal G^\mu](x)\,,
\end{aligned}
\]
and the proof is complete.
\end{proof}

\begin{proof}[Proof of Proposition \ref{teo10}] From Lemmas \ref{lemma8} and \ref{lemma6}, both
$\mathcal G^\mu$ and $ \mathcal L :=\liminf_{n\to \infty} \mathcal G^{\mu_n}$ are $\mathfrak M-$su\-per\-harmonic. Hence, in view of Lemma \ref{lemma7}
and \eqref{e24} we have, for {\em every} $x\in M$,
\[
\begin{aligned}
\mathcal L(x) = & \lim_{r\to 0} \mathfrak M_r[\mathcal L](x)=\lim_{r\to 0}\frac{\alpha+1}{r^{\alpha+1}}\int_0^r \xi^\alpha
\int_{\left\{y\in M\,:\; G(x,y)= 1 / \xi \right\}} \mathcal L(y)
\left|\nabla_y G(x,y)\right| dS(y)d\xi \\
= & \lim_{r\to 0} \frac{\alpha+1}{r^{\alpha+1}} \int_0^r \xi^\alpha
\int_{\left\{y\in M\,:\; G(x,y)= 1 / \xi \right\}} \mathcal
G^\mu(y) \left|\nabla_y G(x,y)\right| dS(y)d\xi = \lim_{r\to 0}
\mathfrak M_r[\mathcal G^\mu](x)=\mathcal G^\mu(x)\,;
\end{aligned}
\]
we point out that here we used \eqref{e26-bis} in order to overcome the fact that $\mathcal G^\mu$ and $\mathcal L$ coincide only $\mathcal V-$ a.e.~in $M$.
\end{proof}

\smallskip

Let us recall the following well-known result, which will be essential in the proof of Theorem \ref{thmexi}, in the case of signed measures.

\begin{lem}[Jordan decomposition]\label{lem-hahn}
Let $ \mu \in \mathcal M_F(M) $. There exists a \emph{unique} couple $ (\mu_+,\mu_-) \in \mathcal M_F^+(M) \times \mathcal M_F^+(M) $ such that $\mu
= \mu_+ - \mu_{-}$ and
\begin{equation}\label{hahn-2}
\mu_{\mathcal{P}} \ge \mu_+ \, , \quad \mu_{\mathcal{N}} \ge \mu_{-}
\end{equation}
for any other couple $ (\mu_{\mathcal{P}} , \mu_{\mathcal{N}} ) \in \mathcal M_F^+(M) \times \mathcal M_F^+(M) $ such that
\begin{equation}\label{hahn-3}
\mu = \mu_{\mathcal{P}} - \mu_{\mathcal{N}} \, .
\end{equation}
Moreover, $ (\mu_+,\mu_-)$ is the \emph{unique minimizer} of the functional
\begin{equation*}\label{hahn-funct}
(\mu_{\mathcal{P}} , \mu_{\mathcal{N}} ) \mapsto \mu_{\mathcal{P}}(M) + \mu_{\mathcal{N}}(M) \quad \textrm{for all } \mu_{\mathcal{P}} ,
\mu_{\mathcal{N}} \in \mathcal M_F^+(M) \ \textrm{subject to \eqref{hahn-3}.}
\end{equation*}
The corresponding minimum is referred to as the \emph{total variation} of $ \mu $, and it is denoted as $ |\mu|(M) $, namely the total mass of the
positive finite Radon measure $ |\mu|=\mu_{+} + \mu_{-} $.
\end{lem}
\begin{proof}
This is a classical result in measure theory, see for instance \cite[Theorem
10.8]{WZ}. We point out that the last statement is just a consequence of \eqref{hahn-2}. In fact, in view of the latter, given any decomposition $
(\mu_{\mathcal{P}} , \mu_{\mathcal{N}}) \neq (\mu_+,\mu_-) $ there necessarily exists a Borel set $ A \subset M $ such that either $
\mu_{\mathcal{P}}(A) > \mu_+(A) $ or $ \mu_{\mathcal{N}}(A) > \mu_-(A) $. In particular,
$$
\begin{aligned}
|\mu|(M)= & \mu_{+}(A)+\mu_+(M \setminus A) + \mu_{-}(A)+\mu_-(M \setminus A) \\
< & \mu_{\mathcal{P}}(A)+\mu_{\mathcal{P}}(M \setminus A) + \mu_{\mathcal{N}}(A)+\mu_{\mathcal{N}}(M \setminus A) = \mu_{\mathcal{P}}(M) +
\mu_{\mathcal{N}}(M) \, .
\end{aligned}
$$
\end{proof}

\begin{oss}\label{hahn-dens}
In the case where $ d\mu(x) = f(x)d\mathcal{V}(x) $ for some $ f \in L^1(M) $, one has $ d\mu_{+}(x)=f_{+}(x)d\mathcal{V}(x)$ and $ d\mu_{-}(x)=
f_{-}(x)d\mathcal{V}(x) $.
\end{oss}

We now show a standard uniqueness result involving potentials of finite Radon measures.
\begin{lem}\label{Lemma-2}
Let $ \mu,\nu \in \mathcal{M}_F $, and suppose that $ \mathcal G^{\mu}(x) = \mathcal G^{\nu}(x) $ for almost every $ x \in M $. Then $ \mu=\nu $. In
particular, if $ \mu $ (or $\nu$) is positive, $ \mathcal G^{\mu}(x) = \mathcal G^{\nu}(x) $ for \emph{every} $ x \in M $.
\end{lem}
\begin{proof}
Let $\phi\in C^\infty_c(M)$. In view of the assumptions, we have
\begin{equation}\label{e103}
\int_M \mathcal G^\mu(x)\Delta \phi(x) \, d\mathcal V(x)\,=\, \int_M \mathcal G^\nu(x) \Delta \phi(x) \, d\mathcal V(x)\,.
\end{equation}
By Fubini's theorem (recall \eqref{e16}), \eqref{e103} is equivalent to
\[
\int_M \int_M G(x,y) \Delta\phi(x) \, d\mathcal V(x)  d\mu(y)\,=\, \int_M \int_M G(x,y) \Delta\phi(x) \, d\mathcal V(x)  d\nu(y) \, ,
\]
that is
\begin{equation}\label{e104}
\int_M \phi(y) \, d\mu(y) \,=\, \int_M \phi(y) \, d\nu(y) \, .
\end{equation}
From \eqref{e104} the thesis follows thanks to density of $C^\infty_c(M)$ in $C_c(M)$.
\end{proof}

\smallskip

The following result, which is crucial for the sequel, is concerned with integrability properties of potentials of functions in $L^1(M)\cap
L^\infty(M)$.

\begin{lem}\label{Lemma0}
Let $N\ge3$ and $f \in L^1(M) \cap L^\infty(M) $. Then $\mathcal G^f \in C(M) \cap L^p(M) $ for all $ p \in (N/(N-2),\infty] $, $
\nabla{\mathcal G^f} \in [L^2(M)]^N $, and the identity
\begin{equation}\label{e-new-1}
\int_{M} \left|\nabla{\mathcal G^f}\right|^2 d\mathcal V = \int_{M} f \, \mathcal G^f \, d\mathcal V
\end{equation}
holds.
\end{lem}
\begin{proof}
We first show that $\mathcal G^f\in C(M)$. To this aim, fix any $x_0\in M$, $\varepsilon>0$, and suppose that $x\in B_\varepsilon(x_0)$. We have:
\[
\begin{aligned}
\left| \mathcal{G}^f(x) - \mathcal{G}^f(x_0) \right| \le & \int_M \left| G(x,y)- G(x_0, y)\right||f(y)| \, d\mathcal V(y) \\
\leq & \| f \|_\infty \int_{B_\varepsilon(x_0)} [G(x, y)+ G(x_0, y)] \, d\mathcal V(y) \\
& + \int_{B^c_\varepsilon(x_0)}\left|G(x, y)-G(x_0, y) \right| |f(y)|\,d\mathcal V(y)\,.
\end{aligned}
\]
Due to \eqref{e16}, since $f\in L^1(M)$, by dominated convergence we get
\[\int_{B^c_\varepsilon(x_0)}\left|G(x, y)-G(x_0, y)
\right||f(y)|\,d\mathcal V(y)\to 0\quad \textrm{as}\ x\to x_0\,.\] On the other hand, since $y\mapsto G(x,y)$ is bounded e.g.~in
$L^{\frac{N-1}{N-2}}(B_\varepsilon(x_0))$ uniformly with respect to $x\in M$ (recall again \eqref{e16}  and the fact that the Riemannian measure $\mathcal V$ is locally Euclidean), and $G(x,y)\to G(x_0, y)$ as $x\to x_0$ for every $y\in M$, we have
that $G(x,y)$ converges weakly to $G(x_0, y)$ in $L^{\frac{N-1}{N-2}}(B_\varepsilon(x_0))$, so that
\[\int_{B_\varepsilon(x_0)}G(x, y) \, d\mathcal V(y)\to \int_{B_\varepsilon(x_0)} G(x_0, y) \, d\mathcal V(y) \quad \textrm{as}\ x\to x_0 \,. \]
Hence,
\[\limsup_{x\to x_0} \left| \mathcal{G}^f(x) - \mathcal{G}^f(x_0) \right|
\leq 2 \, \|f\|_\infty\int_{B_\varepsilon(x_0)}G(x_0, y) \, d\mathcal V(y) \, ,
\]
and the claim follows by letting $\varepsilon \to 0$, thanks to the local integrability of $y \mapsto G(x_0,y)$.

\smallskip
In order to prove that $\mathcal G^f \in L^p(M) $ for all $ p \in (N/(N-2),\infty] $, it is convenient to use the
representation formula \eqref{e35} for the Green function. In fact, by means of \eqref{e33}, \eqref{e31} and interpolation, it is
straightforward to infer the following estimate:
\begin{equation}\label{e-new-3}
\| T_t f \|_p \le C \, t^{-\frac{N(p-1)}{2p}} \, \| f \|_1 \quad \forall p \in (1,\infty) \, , \ \forall t>0 \, ,
\end{equation}
where $C$ is a suitable positive constant depending only on $N $, $p$. As a consequence of \eqref{e35}, we have:
\begin{equation}\label{e-new-4}
\| \mathcal G^f \|_p \leq \int_0^\infty \|T_t f \|_p \, dt = \int_0^1 \|T_t f \|_p \, dt + \int_1^\infty \|T_t f \|_p \, dt \quad \forall p \in [1,\infty] \, .
\end{equation}
By using \eqref{e33} and the fact that $ f \in L^1(M) \cap L^\infty(M) $, it is apparent that the first integral in the r.h.s.~of \eqref{e-new-4} is
finite for every $ p \in [1,\infty] $. By means of \eqref{e31} we can deduce that the second integral in the r.h.s.~of
\eqref{e-new-4} is finite for $p=\infty$; furthermore, thanks to \eqref{e-new-3}, we find that such integral is also finite for all $ p \in (
N/(N-2),\infty ) $. We have therefore shown that $ \mathcal G^f \in L^p(M) $ for all $ p \in ( N/(N-2),\infty ] $.

We are left with the proof of \eqref{e-new-1}. We assume, with no loss of generality, that $ f \ge 0 $. For any $ R>0 $, we denote by $ \mathcal
G_R^f $ the potential of $ f $ in $B_R$, namely the unique $ H^1_0(B_R) $ solution to
\begin{equation*}\label{e-neq-5}
\begin{cases}
-\Delta v = f & \textrm{in } B_R \, , \\
v=0 & \textrm{on } \partial B_R \, .
\end{cases}
\end{equation*}
Clearly,
\begin{equation}\label{e-new-1-a}
\int_{B_R} \left|\nabla{\mathcal G_R^f}\right|^2 d\mathcal V = \int_{B_R} f \, \mathcal G_R^f \, d\mathcal V \, .
\end{equation}
Because $ \mathcal G_R^f $ converges monotonically from below to $ \mathcal G^f $, $ f \in L^1(M) \cap L^\infty(M) $ and from the first part of the
proof we know that $\mathcal G^f \in L^p(M)$ for all $ p \in (N/(N-2),\infty] $, we can pass to the limit in \eqref{e-new-1-a} as $ R\to+\infty $ to
get
\begin{equation*}\label{e-new-1-b}
\int_{M} \left|\nabla{\mathcal G^f}\right|^2 d\mathcal V \le \int_{M} f \, \mathcal G^f \, d\mathcal V
\end{equation*}
and
\begin{equation}\label{e-new-1-c}
\nabla \mathcal G_R^f \rightharpoonup \nabla \mathcal G^f  \quad \textrm{in } [L^2(M)]^N \, ,
\end{equation}
where $ \nabla \mathcal G_R^f $ is set to zero in $ B_R^c $. Exploiting the fact that $ \mathcal G_R^f = 0 $ on $ \partial B_R $ and that $ -\Delta
\mathcal G^f = f $ in $M$, we obtain:
\begin{equation}\label{e-new-1-d}
\int_{B_R} \langle \nabla{\mathcal G_R^f} \, , \, \nabla{\mathcal G^f} \rangle \, d\mathcal V = \int_{B_R} f \, \mathcal G_R^f \, d\mathcal V \, .
\end{equation}
Identity \eqref{e-new-1} then follows by letting $ R \to \infty $ in \eqref{e-new-1-d}, using \eqref{e-new-1-c} and the monotone convergence of $
\mathcal G_R^f $ to $\mathcal G^f$. The case of signed functions follows by writing $f=f^+ - f^-$, and using the linearity of the potential operator.
\end{proof}

\section{Existence of weak solutions: proofs}\label{sec: exi-proof}

This section is devoted to the proofs of our main results concerning existence and fundamental properties of the weak solutions to \eqref{e64} we construct.

\subsection{Consequences of the definition of weak solution}\label{sect-cons}

The aim of this subsection is to prove the following result, which establishes some fundamental properties enjoyed by weak solutions, in the sense of Definition \ref{defsol1}, as such.
\begin{pro}\label{pro: cons}
Let assumption \eqref{H}-\textrm{(i)} be satisfied. Let $u$ be any function satisfying \eqref{e65}--\eqref{e67}. Then:
\begin{equation}\label{eq: cont-L1}
u \in C((0,\infty);L^1(M)) \, ,
\end{equation}
\begin{equation}\label{eq: cont-mass}
\int_M u(x,t_1) \, d\mathcal V(x) = \int_M u(x,t_2) \, d\mathcal V(x) \quad \textrm{for all}\;\; t_2 > t_1 >0 \, ,
\end{equation}
\begin{equation}\label{e70-pre}
\|u(t)\|_\infty \leq K \, t^{-\alpha} \left\| u \right\|_{L^\infty((0,\infty);L^1(M))}^{\beta} \quad \textrm{for all}\;\; t>0 \, ,
\end{equation}
where $K$ is a positive constant which only depends on $m,N$ and $ \alpha,\beta $ are as in \eqref{e71-pre}.
\end{pro}
In order to prove Proposition \ref{pro: cons} we need a preliminary lemma, which relies on results on the porous medium equation that are by now well known.
\begin{lem}\label{lem: pre}
Let $ \mu \equiv u_0 \in L^1(M) \cap L^\infty(M) $. Then there exists a unique weak solution $ u $ to problem \eqref{e64} satisfying \eqref{e65}--\eqref{e66} down to $ \tau=0 $ and
\begin{equation}\label{e67-energy}
-\int_0^\infty\int_M u(x,t)\varphi_t(x,t) \, d\mathcal V(x) dt +
\int_0^\infty\int_M \langle \nabla (u^m)(x,t), \nabla
\varphi(x,t)\rangle \, d\mathcal V(x) dt = \int_M u_0(x) \varphi(x,0) \, d\mathcal V(x)
\end{equation}
for any $\varphi\in C^\infty_c(M\times [0,\infty))$. Moreover, $ u \in C([0,\infty);L^1(M)) $, and if $ v $ is another weak solution to problem \eqref{e64} with initial datum $ v_0 \in L^1(M) \cap L^\infty(M) $ there holds
\begin{equation}\label{eq: L1-contract}
\left\| u(t)-v(t) \right\|_1 \le \left\| u_0-v_0 \right\|_1 \quad \textrm{for all}\;\; t > 0 \, .
\end{equation}
\end{lem}
\begin{proof}
As $ u_0 \in L^1(M) \cap L^\infty(M) $, existence of the so-called energy solutions, namely solutions for which \eqref{e65}--\eqref{e66} hold down to $ \tau=0 $ and \eqref{e67-energy} is satisfied, is rather standard (we refer e.g.~to \cite[Sections 5, 9]{Vaz07} for the Euclidean case). The simplest way to construct them is e.g.~by using approximate problems on balls, establishing suitable a priori estimates and then passing to the limit as the radius of the ball goes to infinity. A sketch of an analogous procedure is provided in the beginning of the proof of Theorem \ref{thmexi} below. Uniqueness in this class is due to a well-known theorem by Ole\u{\i}nik, see \cite[Section 5.3]{Vaz07}. The continuity of $ u(t) $ as a curve in $ L^1(M) $ is then a consequence of an alternative construction of the solution, which makes use of the Crandall-Liggett Theorem and proceeds by means of time discretization (see \cite[Section 10]{Vaz07}). Also the $ L^1 $-contractivity inequality \eqref{eq: L1-contract} is a classical fact (see \cite[Section 3]{Vaz07}).

For similar issues involving existence, uniqueness and equivalence of different concepts of solution (in the framework of the fractional porous medium equation), we also refer to \cite[Appendix A]{GMPas}.
\end{proof}

\begin{proof}[Proof of Proposition \ref{pro: cons}]
Given almost every $t_0>0$, namely any Lebesgue point of $ u(t) $ as a curve in $L^1(M)$, let $\big\{\theta_\varrho^{t_0}\big\}\,\,
(0<\varrho<t_0)$ be a family of positive, smooth
approximations of $\chi_{[t_0, \infty)}$ such that
$\operatorname{supp}\, \theta_\varrho^{t_0} \subset [t_0-\varrho,\infty)$ and $\big(\theta_\varrho^{t_0}\big)^\prime \to
\delta_{t_0}$ as $\varrho\to 0$. Let $ \varphi$ be any function in $ C^\infty_c(M \times [t_0,\infty) ) $: we can assume that $ \varphi $ is the restriction to $ M \times [t_0,\infty) $ of some function in $ C^\infty_c(M \times (0,\infty) ) $. Hence, by plugging in \eqref{e67} the test function
\[
\varphi_\varrho(x,t):=\theta_\varrho^{t_0}(t) \varphi(x,t) \quad \forall (x,t) \in M \times (0,\infty)
\]
and letting $ \varrho \to 0 $, we end up with the identity
\begin{equation*}\label{e67-varrho}
\begin{aligned}
& -\int_{t_0}^\infty\int_M u(x,t) \varphi_t(x,t) \, d\mathcal V(x) dt +
\int_{t_0}^\infty\int_M \langle \nabla (u^m)(x,t), \nabla
\varphi(x,t)\rangle \, d\mathcal V(x) dt \\
= & \int_M u(x,t_0) \varphi(x,t_0) \, d\mathcal V(x) \, .
\end{aligned}
\end{equation*}
On the other hand, by Definition \ref{defsol1} it is apparent that \eqref{e65}--\eqref{e66} hold for $ \tau=t_0 $: we have therefore shown that $ u \rfloor_{[t_0,\infty)} $ is a weak solution to \eqref{e64} (with $ 0 $ replaced by $t_0$) in the sense of Lemma \ref{lem: pre}, starting from the initial datum $ \mu \equiv u(t_0) \in L^1(M) \cap L^\infty(M) $. In particular $ u \in C([t_0,\infty); L^1(M)) $, whence \eqref{eq: cont-L1} because $ t_0 $ can be arbitrarily small.

In order to establish \eqref{eq: cont-mass}, we exploit a reasoning similar to the one outlined in Remark \ref{rem: barenblatt}. Indeed, by the same arguments, we know that the free-mass time-shifted Barenblatt functions
\begin{equation}\label{shif-baren}
B^{E,D}_{0}(\rho(x), t) := \left(t+1\right)^{-\alpha} \left[ D - k \, \rho(x)^2 \left( t+1 \right)^{-\beta} \right]_+^{\frac{1}{m-1}} \quad  \textrm{for all}\;\; (x,t) \in M \times (0,\infty) \, , \ \, \forall D > 0
\end{equation}
are (weak) supersolutions to \eqref{e64} with initial datum $ \mu \equiv B^{E,D}_{0}(\rho(x), 0) $, where $ \alpha,\beta $ are as in \eqref{e71-pre} and $ k $ is a positive constant depending only on $m,N$. Let us first prove \eqref{eq: cont-mass} under the additional assumption that $ u(t_1) $ is compactly supported. In this case, we can always choose $ D $ in \eqref{shif-baren} so large that $ |u(x,t_1)| \le B^{E,D}_{0}(\rho(x), 0) $. Hence, because $ B^{E,D}_{0}(\rho(x), t) $ and $ -B^{E,D}_{0}(\rho(x), t) $ are a supersolution and a subsolution, respectively, it follows that
\begin{equation}\label{shif-baren-sdwc}
-B^{E,D}_{0}(\rho(x), t-t_1) \le u(x,t) \le B^{E,D}_{0}(\rho(x), t-t_1)  \quad \textrm{for a.e. } (x,t) \in M \times (t_1,\infty) \, .
\end{equation}
Since $ B^{E,D}_{0}$ is compactly supported for all times, estimate \eqref{shif-baren-sdwc} implies that $ u $ is also compactly supported for all times. In particular, \eqref{eq: cont-mass} holds. In the case where $ u(t_1) $ is not compactly supported, we can pick a sequence of initial data $ u_{1,n} \in L^1(M) \cap L^\infty(M) $, with compact support, such that $ \lim_{n \to \infty} u_{1,n} = u(t_1) $ in $ L^1(M) $. If we denote by $ u_n $ the solutions to \eqref{e64} corresponding to $ \mu \equiv u_{1,n} $, thanks to the above argument we can deduce that \eqref{eq: cont-mass} is satisfied with $ u $ replaced by $u_n(t-t_1)$: on the other hand, the $L^1$-contractivity inequality \eqref{eq: L1-contract} ensures that the solution map is continuous in $ L^1(M) $, so that we can pass to the limit as $ n \to \infty $ to get \eqref{eq: cont-mass}.

Let us finally deal with the smoothing effect \eqref{e70-pre}. For initial data $ u_0 $ and corresponding solutions $u$ as in Lemma \ref{lem: pre}, the estimate
\begin{equation}\label{e70-proof}
\|u(t)\|_\infty \leq K \, t^{-\alpha} \left\| u_0 \right\|_{L^1(M)}^{\beta} \quad \textrm{for all}\;\; t>0
\end{equation}
holds as a consequence of the Sobolev inequality \eqref{e12} (see e.g.~\cite[Theorem 4.1]{BGV-07} or \cite[Corollary 5.6]{GM13}). Hence, by applying \eqref{e70-proof} to $ u \rfloor_{[t_1,\infty)} $ we obtain
\begin{equation*}\label{e70-proof-bis}
\|u(t)\|_\infty \leq K \, t^{-\alpha} \left\| u(t_1) \right\|_{L^1(M)}^{\beta} \le K \, t^{-\alpha} \left\| u \right\|_{L^\infty((0,\infty);L^1(M))}^{\beta} \quad \textrm{for all}\;\; t>t_1 \, ;
\end{equation*}
since $ t_1>0 $ is arbitrary, the thesis follows.
\end{proof}

\subsection{Proof of the existence result}\label{proof-ex}
Let us outline the main ideas behind the proof of Theorem \ref{thmexi}. Suppose first that $\mu$ is a compactly supported measure. Take $\mu_\varepsilon \in L^1(M)\cap L^\infty(M)$ such that
\begin{equation}\label{e118b}
\int_M \phi \mu_\varepsilon \, d\mathcal V \to \int_{M} \phi \, d\mu  \quad
\textrm{as}\ \varepsilon\to 0 \, , \ \textrm{for any}\ \phi\in C_b(M)
\end{equation}
and
\begin{equation}\label{e118c}
\int_M \left| \mu_\varepsilon \right| d\mathcal{V} \to |\mu|(M) \quad \textrm{as}\ \varepsilon \to 0 \, ;
\end{equation}
to this aim it suffices, for instance, to mollify the image of $\mu$ on $ \mathbb{R}^N $ and then come back to $M$ through one of the regular bijections between $M$ and $\mathbb{R}^N$. For any fixed $\varepsilon>0$ and $R>0$, consider then the following homogeneous Dirichlet problem:
\begin{equation}\label{e90}
\begin{cases}
u_t \,=\, \Delta(u^m) & \textrm{in}\ B_R\times (0,\infty) \, , \\
u\,=\,0 & \textrm{on}\ \partial B_R\times (0,\infty) \, , \\
u \,=\, \mu_\varepsilon \rfloor_{B_R} & \textrm{on}\ B_R\times \{0\} \, ,
\end{cases}
\end{equation}
for which one can provide the same definition of weak solution as in Lemma \ref{lem: pre} upon replacing $M$ with $B_R$ and requiring in addition that $ {u}^m \in H^1_0(B_R) $. Existence, uniqueness and good properties of the weak (energy) solution to \eqref{e90}, which will be denoted by $u_{\varepsilon, R}$, can be shown by means of well-established methods (see again the proof of Lemma \ref{lem: pre}). Classical compactness arguments ensure that $\{u_{\varepsilon,R}\}$ converges (up to subsequences), as $ R \to \infty $, to a
function $u_\varepsilon$ satisfying \eqref{e65}-\eqref{e67}. A further passage to the limit, as $ \varepsilon \to 0 $, yields a function $u$ which still complies with \eqref{e65}-\eqref{e67}. The hardest point consists in proving that $u$ also fulfils \eqref{e68}, namely that its initial trace is precisely $\mu$. To this end we have to adapt to our framework some potential techniques first introduced by M. Pierre in \cite{Pierre} and then recently developed in \cite{Vaz14,GMPmu} in the nonlocal Euclidean context. Finally, we handle general finite measures (i.e.~not necessarily compactly supported) by an additional approximation.

%
\begin{proof}[Proof of Theorem \ref{thmexi}]
By means of standard arguments one can infer that the weak (energy solution) $ u_{\varepsilon, R} $ to \eqref{e90} complies with the non expansivity of the $ L^1 $ norms
\begin{equation}\label{e95cc}
\left\| u_{\varepsilon, R}(t) \right\|_{L^1(B_R)} \leq \left\| \mu_\varepsilon \right\|_{L^1(B_R)} \quad \textrm{for all}\;\; t>0 \, ,
\end{equation}
the $ L^1$-$L^\infty $ smoothing effect
\begin{equation}\label{e95b}
\left\| u_{\varepsilon, R}(t) \right\|_{L^\infty(B_R)} \leq K \, t^{-\alpha} \left\| \mu_\varepsilon \right\|_{L^1(B_R)}^{\beta} \quad \textrm{for all}\;\; t>0
\end{equation}
and the energy estimates
\begin{equation}\label{e114}
\int_{t_1}^{t_2} \int_{B_R} \left|\nabla\!\left(u_{\varepsilon, R}^m \right) \! (x,t)\right|^2 d\mathcal
V(x) dt + \int_{B_R} \left| u_{\varepsilon,R}(x,t_2) \right|^{m+1} d\mathcal V(x) \leq K^m \, t_1^{-\alpha m} \left\| \mu_\varepsilon \right\|_{L^1(B_R)}^{1+\beta m} ,
\end{equation}
\begin{equation}\label{e115}
\int_{t_1}^{t_2} \int_{B_R} \left| (z_{\varepsilon,R})_t (x,t) \right|^2 d\mathcal V(x) dt \leq \widetilde{C} \, t_1^{-\alpha m} \left\| \mu_\varepsilon \right\|_{L^1(B_R)}^{1+\beta m}
\end{equation}
for all $t_2>t_1>0$, where $z_{\varepsilon,R}:=u_{\varepsilon,R}^{(m+1)/2}$ and $\widetilde{C}$ is a positive constant that depends on $N,m,t_1,t_2$ but is independent of $\varepsilon,R$. In the Euclidean context estimates \eqref{e95cc}, \eqref{e114}, \eqref{e115} are by now classical: see again \cite{Vaz07}, in particular Section 5 there. The fact that here $ B_R $ is a ball on a Riemannian manifold is inessential. The smoothing effect \eqref{e95b} is then again a direct consequence of the Sobolev inequality \eqref{e12}. 

Let $ \mathcal{G}_{\varepsilon,R}$ be the potential of $u_{\varepsilon,R}$, that is
\[
\mathcal{G}_{\varepsilon,R}(x,t):=\int_M G_R(x,y) u_{\varepsilon,R}(y,t) \, d\mathcal V(y)\quad \textrm{for all } x\in M \, , \ t>0 \, ,
\]
where $G_R$ is the Green function of the Dirichlet Laplacian in $B_R$. We claim that $\mathcal{G}_{\varepsilon,R}$ solves
\begin{equation*}\label{e116}
\left( \mathcal{G}_{\varepsilon,R} \right)_t = -u_{\varepsilon,R}^m \quad \textrm{in}\ B_R \times (0,\infty) \, ,
\end{equation*}
in the sense that
\begin{equation}\label{e117}
\int_{B_R} \mathcal{G}_{\varepsilon,R}(x, t_2)\phi(x) \, d\mathcal V(x) - \int_{B_R}
\mathcal{G}_{\varepsilon,R}(x,t_1)\phi(x) \, d\mathcal V(x) =
-\int_{t_1}^{t_2} \int_{B_R} u^m_{\varepsilon,R}(x,t)\phi(x) \, d\mathcal V(x) dt
\end{equation}
for all $ t_2>t_1>0 $, for any $\phi\in C^\infty_c(B_R)$. Indeed, by standard elliptic regularity, we have that
\begin{equation}\label{eq: pot-test}
\mathcal G_R^\phi(x) := \int_{B_R} G_R(x,y) \phi(y) \, d\mathcal V(y) \in C^\infty_0(B_R) \, .
\end{equation}
Hence, we are allowed to pick the test function $ \varphi(x,t) = \mathcal G_R^\phi(x) \, [\theta_\varrho^{t_1}(t) - \theta_\varrho^{t_2}(t)] $ in the weak formulation of \eqref{e90}, with $ \theta_\varrho^{t_\cdot} $ defined as in the proof of Proposition \ref{pro: cons}. By using the fact that $ (-\Delta)\mathcal G_R^\phi = \phi $ in $ B_R $, integrating by parts and letting $ \varrho \to 0 $, we get the identity
\begin{equation}\label{e120-BR}
\int_{B_R} u_{\varepsilon,R}(x, t_2) \mathcal G_R^\phi(x) \, d\mathcal V(x) - \int_{B_R}
u_{\varepsilon,R}(x, t_1) \mathcal G_R^\phi(x) \, d\mathcal V(x) = -\int_{t_1}^{t_2} \int_{B_R} u_{\varepsilon,R}^m(x,t) \phi(x) \, d\mathcal V(x) dt \, ,
\end{equation}
namely \eqref{e117} up to an application of Fubini's Theorem in the left-hand side. By letting $ t_1 \to 0 $ in \eqref{e120-BR}, we obtain
\begin{equation}\label{e120-BR-0}
\int_{B_R} u_{\varepsilon,R}(x, t_2) \mathcal G_R^\phi(x) \, d\mathcal V(x) - \int_{B_R}
\mathcal G_R^\phi(x) \, \mu_\varepsilon(x) d\mathcal V(x) = -\int_{0}^{t_2} \int_{B_R} u_{\varepsilon,R}^m(x,t) \phi(x) \, d\mathcal V(x) dt \, .
\end{equation}
We now let $ R \to \infty $. Thanks to \eqref{e95cc}--\eqref{e115}, routine compactness and lower-semicontinuity arguments ensure that $ \{ u_{\varepsilon,R} \} $, set to be zero outside $B_R$, converges almost everywhere (up to subsequences) to some function $ u_\varepsilon $ which satisfies \eqref{e65}--\eqref{e67} (with $u$ replaced by $ u_\varepsilon $) and the analogues of \eqref{e95cc}--\eqref{e115}:
\begin{equation}\label{e95cc-limit}
\left\| u_{\varepsilon}(t) \right\|_{1} \leq \left\| \mu_\varepsilon \right\|_{1} \quad \textrm{for all}\;\; t>0 \, ,
\end{equation}
\begin{equation}\label{e95b-limit}
\left\| u_{\varepsilon}(t) \right\|_{\infty} \leq K \, t^{-\alpha} \left\| \mu_\varepsilon \right\|_{1}^{\beta} \quad \textrm{for all}\;\; t>0 \, ,
\end{equation}
\begin{equation}\label{e114-limit}
\int_{t_1}^{t_2} \int_{M} \left|\nabla\!\left(u_{\varepsilon}^m \right) \! (x,t)\right|^2 d\mathcal
V(x) dt + \int_{M} \left| u_{\varepsilon}(x,t_2) \right|^{m+1} d\mathcal V(x) \leq K^m \, t_1^{-\alpha m} \left\| \mu_\varepsilon \right\|_{1}^{1+\beta m} ,
\end{equation}
\begin{equation}\label{e115-limit}
\int_{t_1}^{t_2} \int_{M} \left| (z_{\varepsilon})_t (x,t) \right|^2 d\mathcal V(x) dt \leq \widetilde{C} \, t_1^{-\alpha m} \left\| \mu_\varepsilon \right\|_{1}^{1+\beta m}
\end{equation}
for all $t_2>t_1>0$, where $z_{\varepsilon}:=u_{\varepsilon}^{(m+1)/2}$. Moreover, since
\begin{equation*}\label{eq: bound-pot-phi}
\lim_{R \to \infty} \mathcal{G}_R^\phi(x) = \mathcal{G}^\phi(x) \quad \forall x \in M \, , \quad  \mathcal{G}^\phi \in C_0(M) \, , \quad \left| \mathcal{G}_R^\phi \right| \le \mathcal{G}^{|\phi|}  \in C_0(M)
\end{equation*}
(consequences of definition \eqref{e21-bis} plus \eqref{e16}, \eqref{e7}, \eqref{e57}), by exploiting estimates \eqref{e95cc}--\eqref{e95b} we can pass to the limit in \eqref{e120-BR-0} to get
\begin{equation*}\label{e120-M}
\int_{M} u_{\varepsilon}(x, t_2) \mathcal G^\phi(x) \, d\mathcal V(x) - \int_{M}
\mathcal G^\phi(x) \, \mu_\varepsilon(x) d\mathcal V(x) = -\int_{0}^{t_2} \int_{M} u_{\varepsilon}^m(x,t) \phi(x) \, d\mathcal V(x) dt \, .
\end{equation*}
As a final step, we let $ \varepsilon \to 0 $. In view of \eqref{e95cc-limit}--\eqref{e115-limit} and \eqref{e118b}--\eqref{e118c}, proceeding as above we deduce that $ \{ u_\varepsilon \} $ converges almost everywhere (up to subsequences) to some function $ u $ which satisfies \eqref{e65}--\eqref{e67},
\begin{equation}\label{e95cc-limit-last}
\left\| u(t) \right\|_{1} \leq \left| \mu \right|\!(M) \quad \textrm{for all}\;\; t>0 \, ,
\end{equation}
\begin{equation}\label{e95b-limit-last}
\left\| u(t) \right\|_{\infty} \leq K \, t^{-\alpha} \left| \mu \right|\!(M)^{\beta} \quad \textrm{for all}\;\; t>0
\end{equation}
and
\begin{equation*}\label{e120-M-last}
\int_{M} u(x, t_2) \mathcal G^\phi(x) \, d\mathcal V(x) - \int_{M}
\mathcal G^\phi(x) \, d \mu(x) = -\int_{0}^{t_2} \int_{M} u^m(x,t) \phi(x) \, d\mathcal V(x) dt \, ,
\end{equation*}
namely
\begin{equation}\label{e120-M-last-2}
\int_{M} \mathcal G(x,t_2) \phi(x)  \, d\mathcal V(x) - \int_{M}
\mathcal G^\mu(x) \phi(x) \, d\mathcal V(x) = -\int_{0}^{t_2} \int_{M} u^m(x,t) \phi(x) \, d\mathcal V(x) dt
\end{equation}
up to an application of Fubini's Theorem, where we denote by $ \mathcal{G}(t) $ the potential of $u(t)$. In particular, by combining \eqref{e95cc-limit-last}--\eqref{e95b-limit-last} and \eqref{e120-M-last-2} we deduce the estimate
\begin{equation}\label{e125}
\left| \int_M \mathcal{G}(x, t_2) \phi(x) \, d\mathcal V(x) - \int_M \mathcal{G}^{\mu}(x)
\phi(x) \, d \mathcal V(x) \right| \leq \|\phi \|_{\infty} \, K^{m-1} \, |\mu|(M)^{1+\beta(m-1)} \, \int_0^{t_2}t^{-\alpha(m-1)} \, dt \, .
\end{equation}
By compactness results in measure spaces (recall Definition \ref{def1}), from \eqref{e95cc-limit-last} it follows that every sequence $t_n\to 0$ admits a subsequence $\{t_{n_k}\}$ such that
$\{u(t_{n_k})\}$ converges vaguely to a certain finite Radon measure $\nu$. On the other hand, as noted above, $\mathcal G^\phi(x) $ is a continuous function that vanishes as $ \operatorname{dist}(x,o) \to \infty$. We can therefore pass to the limit as $ t_2 \to 0 $ in \eqref{e125} (by using again Fubini's Theorem): because $ \phi $ is arbitrary, it follows that $\mathcal G^\mu=\mathcal G^\nu$ almost everywhere in $M$; so, thanks to Lemma \ref{Lemma-2}, we have that $\nu=\mu$ and the limit measure does not depend on the particular subsequence. We have thus proved that
\begin{equation}\label{e68-proof}
\lim_{t\to 0} \int_M u(x,t) \phi(x) \, d\mathcal V(x) \,=\, \int_{M} \phi(x) \, d\mu(x) \quad \textrm{for any}\ \phi \in C_0(M) \, .
\end{equation}
In particular, given the lower semicontinuity of the total variation w.r.t.~the vague topology, 
\begin{equation}\label{e128}
|\mu|(M) \leq \liminf_{t\to 0} \| u(t) \|_1 \, ,
\end{equation}
so that by gathering \eqref{e95cc-limit-last} and \eqref{e128} we obtain
\begin{equation}\label{e129}
\lim_{t\to 0} \| u(t) \|_1 = |\mu|(M) \, .
\end{equation}
We are then left with proving that \eqref{e68-proof} holds for any $ \phi \in C_b(M) $. To this aim, we exploit Lemma \ref{lem-hahn}.
In fact, by \eqref{e95cc-limit-last} and \cite[Theorem 1.59]{AFP}, given any sequence $t_n\to 0$ there exists a subsequence $\{t_{n_k}\}$ such that
$\{u_+(t_{n_k})\}$ and $\{u_-(t_{n_k})\}$ converge vaguely to some positive finite Radon measures $\mu_{\mathcal P}$ and $ \mu_{\mathcal N} $,
respectively. Thanks to \eqref{e68-proof} it follows that $\mu = \mu_{\mathcal P} - \mu_{\mathcal N}$. Moreover, as a consequence of \eqref{e129}
and of the lower semicontinuity of the total variation w.r.t.~the vague topology, we have:
\begin{equation}\label{hahn-6-p}
\mu_{\mathcal P}(M) + \mu_{\mathcal N}(M) \le \liminf_{k\to\infty} \left\| u_+(t_{n_k}) \right\|_1 + \liminf_{k\to\infty} \left\| u_-(t_{n_k}) \right\|_1  \le \liminf_{k\to\infty} \left\| u(t_{n_k}) \right\|_1 = |\mu|(M) \, .
\end{equation}
By Lemma \ref{lem-hahn}, \eqref{hahn-6-p} implies $ \mu_{\mathcal P}=\mu_+ $ and $ \mu_{\mathcal N} = \mu_- $, so that
\begin{equation}\label{hahn-7-p}
\lim_{k\to\infty} \left\| u_+(t_{n_k}) \right\|_1 = \mu_{+}(M)  \, , \quad \lim_{k\to\infty} \left\| u_-(t_{n_k}) \right\|_1 = \mu_{-}(M) \, .
\end{equation}
Due to \eqref{hahn-7-p} and \cite[Proposition 1.80]{AFP} we can then infer that
\begin{equation*}\label{hahn-4-p}
\lim_{k\to \infty} \int_M u_{\pm}(x,t_{n_k}) \phi(x) \, d\mathcal V(x) = \int_{M}\phi(x) \, d\mu_{\pm}(x)
\end{equation*}
for all $ \phi \in C_b(M) $. Since the same argument can be performed along any sequence, the validity of \eqref{e68} follows. Note that the conservation of ``mass'' \eqref{e69} is an immediate consequence of \eqref{eq: cont-mass} and \eqref{e68} with the choice $ \phi = 1 $.

\smallskip
Finally, in order to handle a general finite Radon measure $\mu$ (i.e.~not necessarily compactly supported), it is enough to approximate $ \mu $ with the sequence $ \{ \mu\rfloor_{B_n} \} $ as $ n \to \infty $, and proceed in a similar way as above.
\end{proof}

\begin{oss}\label{supercritical}
\rm[\it The case $ \frac{N-2}{N} < m < 1 $\rm]
By using the same techniques as in the proof of Theorem \ref{thmexi}, we can establish existence of weak solutions to problem \eqref{e64} also for $m$ below $1$, in the supercritical fast-diffusion range $ \frac{N-2}{N} < m < 1 $. Indeed, well-posedness of the approximate problems \eqref{e90} still holds, as well as the key estimates \eqref{e95cc}--\eqref{e115}: the assumption $ m > \frac{N-2}{N} $ plays a crucial role in the validity of the smoothing effect \eqref{e95b} (see \cite[Theorem 4.1]{BGV-07}). The only difference lies in the fact that, since $ m < 1 $, the r.h.s.~of \eqref{e120-M-last-2} has to be bounded as follows:
$$ \left| \int_{0}^{t_2} \int_{M} u^m(x,t) \phi(x) \, d\mathcal V(x) dt \right| \le \left\| \phi \right\|_\infty \mathcal{V}\left( \operatorname{supp} \phi \right)^{1-m}  |\mu|(M)^m \, t_2 \, . $$
Actually, the only point that we are not able to recover in Theorem \ref{thmexi} is the conservation of ``mass'' \eqref{e69}. The problem is that, for $m$ smaller than $1$, the analogues of the Euclidean Barenblatt profiles \eqref{shif-baren} we exploit in the proof of Proposition \ref{pro: cons} are no more compactly supported, and their decay rate at infinity is too slow compared to the possible volume growth of the Riemannian manifolds we are interested in. On the other hand, in general mass conservation fails: for instance, on Riemannian manifolds supporting the Poincar\'e/gap inequality $ \| f \|_2 \le \| \nabla f \|_2 $ for all $ f \in C^\infty_c(M) $ (like those whose sectional curvatures are bounded from above by a negative constant), the $ L^1(M) $ norm of the solution vanishes after a finite time \cite[Theorem 6.1]{BGV-07}.
\end{oss}

\subsection{Connection between the Green function and the porous medium equation: proof}\label{alternative}
Let us consider again the solutions $ u_{\varepsilon,R} $ to the approximate problems \eqref{e90}. In the case $\mu\in \mathcal M_F^+(M)$ such solutions are by construction nonnegative: hence, by the standard comparison principle, for all $0<R_1<R_2$ there holds
\begin{equation}\label{e205}
u_{\varepsilon, R_1}(x,t)\, \leq \, u_{\varepsilon, R_2}(x,t) \quad \textrm{for a.e.}\ (x,t)\in B_{R_1}\times (0, \infty)\,.
\end{equation}
For any fixed $ R>0 $, if we let $ \varepsilon \to 0 $ we obtain, by means of the same techniques of proof of Theorem \ref{thmexi}, a nonnegative weak solution $ u_R $ to \eqref{e90} with $ \mu_\varepsilon $ replaced by $ \mu $. By letting $ \varepsilon \to 0 $ in \eqref{e205} we also deduce that order is preserved, namely
\begin{equation*}\label{e206}
u_{R_1}(x,t)\,\leq\, u_{R_2}(x,t) \quad \textrm{for a.e.}\ (x,t)\in B_{R_1}\times (0,\infty)\,,
\end{equation*}
namely the family $\{u_R\}$ is nondecreasing in $R$. As a consequence, the pointwise limit $u$ as $ R \to \infty $ exists regardless of the validity of hypothesis \eqref{H}: in such general framework, this is precisely what we mean as a ``solution'' to \eqref{e64} when $ \mu $ is a positive measure.

\smallskip

\begin{proof}[Proof of Theorem \ref{thmBar}] Let $u_R$ be the solution of problem \eqref{e90} with $\mu_\varepsilon\equiv\delta_{x_0}$,
where $ R $ is supposed to be so large that $ x_0 \in B_R $. Let us denote by $\mathcal G_R$ the potential of $u_R$ and by $\mathcal G_R^\phi$ the potential of any $ \phi \in C_c^\infty(B_R)$ (recall \eqref{eq: pot-test}). Given any $ t_2>t_1>0 $, by plugging
the test function $ \varphi(x,t)=\mathcal G_R^\phi(x)[\theta_\varrho^{t_1}(t) - \theta_\varrho^{t_2}(t)] $ in the definition of weak solution ($ \theta_\varrho^{t_\cdot} $ is as in the proof of Proposition \ref{pro: cons}), letting $\varrho \to 0 $ and
exploiting Tonelli's theorem, we end up with the identity
\begin{equation}\label{e212}
\int_{B_R} \mathcal{G}_R (x, t_2) \phi(x) \, d\mathcal V(x) - \int_{B_R} \mathcal{G}_R (x, t_1) \phi(x) \, d\mathcal V(x) \,=\, - \int_{t_1}^{t_2}
\int_{B_R} u_R^m(x,t) \phi(x) \, d\mathcal V(x) dt \, .
\end{equation}
From \eqref{e212} we deduce that the map $ t \mapsto \mathcal{G}_R (x, t) $ is nonincreasing (recall that $  u_R $ is nonnegative). Hence, $ \mathcal{G}_R (t) $ admits a pointwise limit
as $ t \to \infty $. Such limit is necessarily zero: this is a straightforward consequence, for instance, of the smoothing estimate \eqref{e70}, 
which clearly holds for \eqref{e90} as well. Passing to the limit
in \eqref{e212} as $ t_2 \to \infty $ we then get
\begin{equation}\label{e213}
\int_{B_R} \mathcal{G}_R (x, t_1) \phi(x) \, d\mathcal V(x) \,=\, \int_{t_1}^{\infty} \int_{B_R} u_R^m(x,t) \phi(x) \, d\mathcal V(x) dt \, .
\end{equation}
Letting $t_1\to0$ in \eqref{e213}, recalling the initial condition and using again Tonelli's theorem we infer that
\begin{equation}\label{e213-bis}
\int_{B_R} G_R (x_0, x) \phi(x) \, d\mathcal V(x) \,=\, \int_{B_R} \phi(x) \int_{0}^{\infty} u_R^m(x,t) \, dt d\mathcal V(x) \, .
\end{equation}
Now we point out that both $ x \mapsto \int_{0}^{\infty} u_R^m(x,t) \, dt $ and $ x \mapsto G_R (x_0, x) $ are $\mathfrak{M}-$superharmonic
functions belonging to $ L^1(B_R) $. Indeed, in view of standard results concerning the porous medium equation on bounded domains (see the monograph
\cite{Vaz07}), it is well known that $ \| u_R(t) \|_{L^\infty(B_R)}$ behaves at most like $ t^{-N/[2+N(m-1)]} $ as $ t \to 0 $ and at most like $
t^{-1/(m-1)} $ as $t \to \infty$. This immediately implies that $ \int_{0}^{\infty} u_R^m(t) \, dt \in L^1(B_R) $. Moreover, by means of classical
results, we know that $ u(x,t) $ is continuous in $ B_R \times [t_1,\infty) $ for all $t_1>0$. In particular, by dominated convergence, we deduce
that $ \int_{t_1}^{\infty} u_R^m(t) \, dt $ is also continuous for all $ t_1>0 $. As a consequence of the differential equation solved by $u_R$ and
of the fact that $ \| u_R(t) \|_{L^\infty(B_R)} $ vanishes as $ t \to \infty $, we deduce that $ \int_{t_1}^{\infty} u_R^m(t) \, dt $ is
superharmonic. Hence, thanks to Theorem \ref{teo5}, we can claim that $ \int_{t_1}^{\infty} u_R^m(t) \, dt $ is $ \mathfrak{M}-$superharmonic; so, in
view of Lemma \ref{lemma8}, we can infer that $  \int_{0}^{\infty} u_R^m(t) \, dt $ is also $ \mathfrak{M}-$superharmonic. The fact that $ x \mapsto
G_R (x_0, x) $ is $\mathfrak{M}-$superharmonic follows from Corollary \ref{cor10}; in addition, it belongs to $ L^1(B_R) $
since $B_R$ is bounded (see e.g.~\cite{Grig3}). In view of the above remarks, \eqref{e213-bis} and Lemma \ref{lemma7}, there holds
\begin{equation}\label{e214}
G_R (x_0, x)  \,=\, \int_{0}^{\infty} u_R^m(x,t) \, dt  \quad \textrm{for all}\ x \in B_R \, .
\end{equation}
The thesis then follows from \eqref{e214} by monotone convergence, using the fact that $ G_R \uparrow G$ as $R\to \infty$ everywhere.
\end{proof}

\section{Proof of the uniqueness result}\label{proof-uniq}
We begin this section with a key lemma, which will be very useful in the sequel and which is essentially based on the potential theoretic results given in Sections \ref{Pot-MVP} and \ref{Pot}. To our purposes it is crucial that the limit in \eqref{e131} below is taken for \it every \rm $x$, this following from Proposition \ref{teo10}.
\begin{lem}\label{Lemma11}
Let $u$ be a nonnegative weak solution of problem \eqref{e64} with
$\mu\in \mathcal M^+_F(M)$. Then the potential $\mathcal G(t)$ of
$u(t)$ satisfies the following equation:
\begin{equation}\label{e130a}
\mathcal G_t\,=\, - u^m \quad \textrm{in}\ M\times
(0,\infty)\,,
\end{equation}
in the sense that
\begin{equation}\label{e130}
\int_M \mathcal{G}(x, t_2)\phi(x) \, d\mathcal V(x) - \int_M
\mathcal{G}(x,t_1)\phi(x) \, d\mathcal V(x) = -\int_{t_1}^{t_2}
\int_M u^m(x,t)\phi(x) \, d\mathcal V(x) dt
\end{equation}
for all $ t_2>t_1>0 $, for any $\phi\in C^\infty_c(M)$. In
particular, it admits an absolutely continuous version on $(0,\infty)$ in $L^p(M)$ for all $  p \in (N/(N-2),\infty) $, which is nonincreasing in $t$. Moreover,
\begin{equation}\label{e131}
\lim_{t\to 0} \mathcal G(x,t)\,=\, \mathcal G^\mu(x)\quad
\textrm{for all}\ x\in M\,.
\end{equation}
\end{lem}
\begin{proof}
Consider a cut-off function $\xi\in C^\infty([0,\infty))$ with
\[
\xi=\begin{cases}
1 & \textrm{in}\ [0,1]\,, \\
0 & \textrm{in}\ [2,\infty) \, ,
\end{cases}\;\,\quad \textrm{and}\;\; 0\leq \xi \leq 1\quad \textrm{in} \ [0,\infty)\,.
\]
Set $\rho(x):= d(x,o)$ for all $x\in M$. For every $ R \ge 1 $, define
\begin{equation*}\label{e75}
\xi_R(x):=\xi\left(\frac{\rho(x)}{R}\right) \quad \textrm{for
all}\ x\in M\,.
\end{equation*}
Let $C:=\max\{\sup_{[0,\infty)}|\xi'|, \, \sup_{[0,\infty)}|\xi''|\}$. In view of \eqref{e1} we have
\begin{equation}\label{e76}
\Delta \xi_R(x)\,=\,\frac 1{R^2} \, \xi''\left(\frac{\rho(x)}R\right)
+ \frac{m(\rho, \theta)}{R} \, \xi'\left(\frac{\rho(x)}R\right)\quad \textrm{for
all}\ x\in M\,.
\end{equation}
Clearly
\begin{equation}\label{e78}
\nabla{\xi_R} = 0  \ \ \textrm{and} \ \ \Delta \xi_R = 0 \quad \textrm{in } B_R \cup B_{2R}^c \, ;
\end{equation}
moreover,
\begin{equation*}\label{e77-grad}
\left| \nabla{\xi_R}(x) \right| \le \frac{C}{R} \quad \textrm{for all}\;\; x \in B_{2R} \setminus B_R
\end{equation*}
since $ | \nabla \rho (x) | = 1 $. Furthermore, thanks to assumption \eqref{H}-\textrm{(ii)}, it is not difficult to check that there exists a positive constant $ \hat{C} $ such that \eqref{e4} is fulfilled by a suitable $ \psi $ satisfying $ \psi(\rho)=e^{\hat{C}\rho^2} $ for all $ \rho $ large enough. As a consequence, by exploiting also \eqref{e5}, from \eqref{e76} we can infer that
\begin{equation}\label{e77}
\big|\Delta \xi_R(x)\big| \leq \frac{1}{R^2} \left|\xi''\left(\frac{\rho(x)}{R} \right)\right|+\frac{N-1}{R}\left|
\frac{\psi'(\rho(x))}{\psi(\rho(x))}\xi'\left(\frac{\rho(x)}{R} \right)\right| \leq \frac C{R^2} + 2 C\hat{C}(N-1) \leq \bar C \quad \forall x \in B_{2R} \setminus B_R
\end{equation}
for another positive constant $ \bar C$ that depends only on $ N $, $C$ and $ \hat{C} $.

In view of \eqref{e16}, \eqref{e7} and \eqref{e57}, we have that the potential $ \mathcal G^\phi $ of $\phi$, namely
$$  \mathcal G^\phi(x) := \int_{M} G(x,y) \phi(y) \, d\mathcal V(y) \quad \textrm{for all}\ x \in M  \, , $$
is a regular function belonging to $ C_0(M) $. For every $R \ge 1 $ and $\varrho>0$, we are therefore allowed to pick the test function
$$ \varphi(x,t) := \xi_R(x) \, \mathcal G^\phi(x) \left[\theta_\varrho^{t_1}(t) - \theta_\varrho^{t_2}(t) \right] \quad \textrm{for all } x\in M \, , \ t\geq 0 $$
in \eqref{e67}, where $ \theta_\varrho^{t_\cdot} $ is defined as in the proof of Proposition \ref{pro: cons}. By letting $\varrho\to 0$ we get
\begin{equation}\label{e120}
\begin{aligned}
& \int_M u(x, t_2) \xi_R(x) \mathcal G^\phi(x) \, d\mathcal V(x) -\int_M
u(x, t_1) \xi_R(x) \mathcal G^\phi(x) \, d\mathcal V(x) \\
= & \int_{t_1}^{t_2} \int_M u^m(x,t) \Delta( \xi_R \, \mathcal G^\phi )(x) \, d\mathcal V(x) dt \, ;
\end{aligned}
\end{equation}
the r.h.s.~of \eqref{e120} reads (we use the fact that $ -\Delta \mathcal G^\phi = \phi $ in $M$)
\begin{equation}\label{e120-bis}
\begin{aligned}
& -\int_{t_1}^{t_2} \int_M u^m(x,t) \xi_R(x) \phi (x) \, d\mathcal V(x) dt + \underbrace{\int_{t_1}^{t_2} \int_M u^m(x,t) \Delta \xi_R(x) \mathcal G^\phi(x) \, d\mathcal V(x) dt}_{I_1} \\
& + 2 \underbrace{\int_{t_1}^{t_2} \int_M u^m(x,t) \langle \nabla \xi_R \, , \, \nabla \mathcal G^\phi \rangle (x) \, d\mathcal V(x) dt}_{I_2}  \, .
\end{aligned}
\end{equation}
In view of \eqref{e78}--\eqref{e77}, we can estimate the last two integrals of \eqref{e120-bis} as follows:
\begin{equation}\label{e120-ter}
\left|I_1\right| \le \bar{C} \, \| \mathcal G^\phi \|_\infty \, \int_{t_1}^{t_2} \int_{B_{2R}\setminus B_R} u^m(x,t) \, d\mathcal V(x) dt \, ,
\end{equation}
\begin{equation}\label{e120-quater}
\left|I_2\right| \le \frac{C(t_2-t_1)^{\frac12}}{R} \left( \int_{t_1}^{t_2} \int_{B_{2R}\setminus B_R} u^{2m}(x,t) \, d\mathcal V(x) dt \right)^{\frac12} \left( \int_{B_{2R}\setminus B_R} | \nabla \mathcal G^\phi (x) |^2 \, d\mathcal V(x) \right)^{\frac12} \, .
\end{equation}
Since $u \in L^p(M\times (\tau, \infty))$ for every $\tau>0$, $ p\in [1, \infty]$, and $ \nabla{\mathcal{G}^\phi} \in [L^2(M)]^N $ (recall Lemma \ref{Lemma0}), by letting $ R \to \infty$ 
we deduce that $ I_1$ and $I_2 $ vanish so that, by passing to the limit in \eqref{e120} we get
\begin{equation*}\label{e121}
\int_M u(x, t_2) \mathcal G^\phi(x) \, d\mathcal V(x) - \int_M u(x, t_1) \mathcal G^\phi(x) \, d\mathcal V(x) \,=\, - \int_{t_1}^{t_2} \int_M u^m(x,t) \phi(x) \, d\mathcal V(x) dt \, ,
\end{equation*}
namely \eqref{e130} up to an application of Tonelli's Theorem. The absolute continuity of the potential $\mathcal G(t)$ as a curve in $L^p(M)$ for any $ p \in (N/(N-2),\infty) $ is then a consequence of \eqref{e130} and Lemma \ref{Lemma0} (we use the fact that $ u(t)\in L^1(M)\cap L^\infty(M) $). Since $ u \ge 0 $, still by \eqref{e130} and Lemma \ref{Lemma0} we deduce that for \emph{every} $x\in M$ the function $t\mapsto \mathcal G(x,t)$ is nonincreasing.

In order to establish \eqref{e131}, pick a sequence $\{t_n\}\subset
(0,\infty)$ such that $t_n\to 0$ as $n\to \infty.$ Note that from
\eqref{e16} and the fact that $ u \in L^\infty((0,\infty);L^1(M)) $ we can infer that for each compact subset $K\subset M$ and for any $\varepsilon>0$ there exists $R_\varepsilon>0$ such that
\begin{equation*}\label{e132}
\int_{B^c_R} \int_K G(x,y) \, d\mathcal V(y) \, u(x, t_n) \,  d\mathcal V(x)
\le \varepsilon \quad \textrm{for all}\;\; R>R_\varepsilon \, , \ n \in \mathbb N \, .
\end{equation*}
Furthermore, by Definition \ref{defsol1}, we know that $\{u(t_n)\}$ converges vaguely to $\mu$. We can therefore apply Proposition \ref{teo10} to deduce that
\begin{equation*}\label{e133}
\mathcal G^{\mu}(x) \,=\, \liminf_{n\to \infty} \mathcal G(x,
t_n)\quad \textrm{for every}\ x\in M \, .
\end{equation*}
This implies \eqref{e131}, due to the just mentioned monotonicity
property of $ t \mapsto \mathcal G(x,t)$.
\end{proof}

\subsection{Formal strategy of proof}\label{formal-proof}
Our method of proof is modelled after the one given in \cite{Pierre} in the Euclidean context (see also the proof of \cite[Theorem 3.4]{GMPmu}). We sketch it below.

Let $u_1,u_2$ be two weak solutions of problem \eqref{e64} which take on the same initial measure $\mu \in \mathcal M^+_F(M)\,.$ Let $\mathcal
G_1(t)$ and $\mathcal G_2(t)$ be the corresponding potentials. Given any $h>0$, define the function
\begin{equation}\label{e146}
W(x,t):=\mathcal G_2(x, t+h) - \mathcal G_1(x,t)\quad \textrm{for all}\ x\in M \, , \ t>0\,.
\end{equation}
In view of Lemma \ref{Lemma11}, we have that $W(t)$ satisfies
\begin{equation}\label{e134}
W_t(x,t)\,=
a(x,t)\, \Delta W (x,t)\quad \textrm{in}\ M\times (0,\infty)\,,
\end{equation}
where
\begin{equation}\label{e135}
a(x,t):=\begin{cases} \frac{u_1^m(x,t)-u_2^m(x,
t+h)}{u_1(x,t)-u_2(x, t+h)}>0 & \textrm{if}\,\, u_1(x,t)\neq u_2(x,
t+h) \, , \\
0 & \textrm{elsewhere}\,.
\end{cases}
\end{equation}
Still Lemma \ref{Lemma11} yields $W(x,0)\leq 0$ in $M$. The conclusion would follow if we could show that $W\le0$ in $ M \times (0,\infty) $, since this would imply, interchanging the roles of $u_1$ and $ u_2 $, that $W=0$ and hence, letting $ h \to 0 $, that $u_1 = u_2$. In order to prove that $W\le0$ one considers solutions of
the {\em dual} problem
\begin{equation}\label{e136}
\begin{cases}
\varphi_t\,=\,- \Delta (a \varphi) & \textrm{in}\ M\times
(0,T]\,, \\
\varphi\,=\, \psi & \textrm{on}\ M\times\{T\}\,,
\end{cases}
\end{equation}
for any $\psi\in C^\infty_c(M)$ with $\psi\geq 0$ and $T>0$.

Using the solutions of such dual problem as test functions in the weak formulation of \eqref{e134} one formally gets
\begin{equation*}\label{e139bis}
\int_M W(x, T) \psi(x) \, d\mathcal V(x)\,=\, \int_M W(x,0)\varphi(x, 0)\, d{\mathcal V}(x) \le 0 .
\end{equation*}
The claim follows since $\varphi$ is by construction nonnegative. In fact, such procedure must be carefully justified by means of suitable
approximations of problem \eqref{e136}.

\subsection{Existence and basic properties of the approximate solutions $\varphi_{\varepsilon, n}$}\label{basic-eps-n}
For every $n\in \mathbb N$ and $\varepsilon>0$ we consider nonnegative solutions $\varphi_{n,\varepsilon}$ of the problem
\begin{equation}\label{e138}
\begin{cases}
(\varphi_{n,\varepsilon})_t\,=\, -\Delta [(a_n + \varepsilon)\varphi_{n,
\varepsilon}] & \textrm{in}\ M\times (0, T]\,, \\
\varphi_{n, \varepsilon}\,=\, \psi & \textrm{on}\ M\times\{T\}\,,
\end{cases}
\end{equation}
where the sequence $\{a_n\}$ is a suitable approximation of the function $a$ defined by \eqref{e135}. The functions $ \varphi_{\varepsilon, n} $ are constructed by making an appropriate use of linear semigroup theory; in particular, we take advantage of the fact that $-\Delta$ is a positive self-adjoint operator generating a Markov semigroup on $L^2(M)$ (see \cite{Grig}).

The arguments one can exploit in the proof of the forthcoming lemma closely resemble those used to establish \cite[Lemma 5.3]{GMPmu}, hence we skip it.
\begin{lem}\label{Lemma12}
Let $\{a_n\}$ be a sequence of nonnegative functions converging a.e.~to the
function $a$ defined in \eqref{e135} such that:
\begin{itemize}
\item{} for any $n\in \mathbb N$ and $t>0$, $ x \mapsto a_n(x,t)$ is a
regular function; \item{} for any $n\in \mathbb N$ and
$x\in M$, $t \mapsto a_n(x,t)$ is a piecewise constant function,
which is constant on each time interval $(T-(k+1)T/n, \, T-kT/n],$
for every $k\in \{0, \ldots, n-1\}$; \item
$\{\|a_n\|_{L^\infty(M\times (\tau, \infty))}\}$ is uniformly
bounded w.r.t.~$n\in\mathbb{N}$ for any $\tau>0$.
\end{itemize}
Then, for any $\varepsilon>0$ and for any $\psi\in C^\infty_c(M)$
with $\psi\geq 0$, there exists a nonnegative solution
$\varphi_{n, \varepsilon}$ to problem \eqref{e138}, in the sense
that $\varphi_{n, \varepsilon}(t)$ is a continuous curve in
$L^p(M)$ (for all $1<p<\infty$) satisfying $\varphi_{n,
\varepsilon}(T)=\psi$ and it is absolutely continuous on
$(T-(k+1)T/n, \, T-kT/n)$ for each $k\in \{0, \ldots, n-1\}$, so
that the identity
\begin{equation}\label{e141}
\varphi_{n, \varepsilon}(t_2)-\varphi_{n,
\varepsilon}(t_1)\,=\,-\int_{t_1}^{t_2}\Delta[(a_n
+\varepsilon)(\tau) \, \varphi_{n, \varepsilon}(\tau)] \, d\tau
\end{equation}
holds in $L^p(M)$ (for all $1<p<\infty$) for any $t_1, t_2\in
(T-{(k+1)T}/{n}, \, T-{kT}/n )$ and for any $k \in \{0, \ldots, n-1\}$. Moreover,
\begin{equation}\label{e142}
\varphi_{n, \varepsilon}\in L^\infty \left((0,T); L^p(M)\right) \ \,
\textrm{for all}\ p \in [1, \infty] \quad \textrm{and} \quad \|\varphi_{n, \varepsilon}(t)\|_1 \leq \|\psi\|_1 \ \, \textrm{for
all}\ t\in [0, T] \, .
\end{equation}
\end{lem}

\smallskip

In the proofs of the next lemmas, even if we follow the general strategy used to show analogous results in \cite{GMPmu}, there are some additional difficulties to overcome. They are related to the fact that an analogue of \cite[Proposition B.1]{GMPmu} is not available in the present framework, because of a possible different growth of the volume of balls. Thus, more delicate cut-off arguments are required.

\smallskip

We now prove some crucial identities involving the functions $ \varphi_{n,\varepsilon} $ and $ W $.
\begin{lem}\label{Lemma13}
Let $W$ be defined as in \eqref{e146}, $a$ as in \eqref{e135}, and $a_n$, $\varphi_{n, \varepsilon}$, $\psi$ as in Lemma \ref{Lemma12}. Then the identity
\begin{equation}\label{e147}
\begin{aligned}
&\int_M W(x,T) \psi(x) \, d\mathcal V(x) - \int_M W(x,t) \varphi_{n,
\varepsilon}(x,t) \, d\mathcal V(x) \\  = & - \int_t^T\int_M \left[a_n(x,
\tau) + \varepsilon-a(x,\tau)\right] \Delta W(x, \tau)\varphi_{n,
\varepsilon}(x, \tau) \, d\mathcal V(x) d\tau
\end{aligned}
\end{equation}
holds for all $t\in (0,T)$.
\end{lem}
\begin{proof}
Let us set
\begin{equation*}\label{e159} t_k:=\frac{T(n-k)}{n}
\quad \textrm{for all}\ k \in \{0, \ldots, n\}\,.
\end{equation*}
Thanks
to Lemma \ref{Lemma11}, we know that $W(t)$ is an absolutely
continuous curve in $L^p(M)$ for all $p\in (N/(N-2), \infty)$,
satisfying \eqref{e134}. On the other hand Lemma \ref{Lemma12}
ensures that $\varphi_{n, \varepsilon}(t)$ is a continuous curve
in $L^p(M)$ for all $p\in (1, \infty)$ on $(0, T]$, absolutely
continuous on $(t_{k+1}, t_{k})$ for each $k\in \{0, \ldots,
n-1\}$ and satisfying the differential equation in \eqref{e138} on
such intervals. Hence, the function
\begin{equation*}\label{e148}
t \mapsto \int_M  \xi_R(x) W(x,t) \varphi_{\varepsilon, n}(x,t) \, d\mathcal V(x) \, ,
\end{equation*}
where $ \{ \xi_R \}_{R>0} $ is a cut-off family as in the proof of Lemma \ref{Lemma11}, is continuous on $(0, T]$, absolutely continuous on $(t_{k+1}, t_k)$ and satisfies
\begin{equation}\label{e149}
\begin{aligned}
& \frac{d}{dt} \int_M \xi_R(x) W(x,t) \varphi_{n, \varepsilon}(x,t) \, d\mathcal V(x) \\
= & \int_M \left\{\xi_R(x) a(x,t) \Delta W(x,t) \varphi_{n,
\varepsilon}(x,t)-\xi_R(x)W(x,t)\Delta[(a_n+\varepsilon) \varphi_{n, \varepsilon}](x,t) \right\} d\mathcal V(x) \quad \textrm{on}\ (t_{k+1}, t_k)\,.
\end{aligned}
\end{equation}
By standard elliptic regularity, we have that $ W(t) \in W^{2,p}_{\rm loc}(M)  $ for all $ p \in (1,\infty) $. We can therefore integrate by parts the last term in the r.h.s.~of \eqref{e149} to get:
\begin{equation}\label{e149-bis}
\begin{aligned}
& \int_M \xi_R(x) W(x,t) \Delta[(a_n+\varepsilon) \varphi_{n, \varepsilon}](x,t) \, d\mathcal V(x) \\
 = & \int_M \xi_R(x) \Delta W(x,t) [a_n(x,t) + \varepsilon] \varphi_{n, \varepsilon}(x,t) \, d\mathcal V(x) + \underbrace{\int_M \Delta \xi_R(x) W(x,t)  [a_n(x,t) + \varepsilon] \varphi_{n, \varepsilon}(x,t) \, d\mathcal V(x)}_{I_1(t)} \\
 & + 2 \underbrace{\int_M \langle \nabla \xi_R(x),\nabla W(x,t) \rangle \,  [a_n(x,t) + \varepsilon] \varphi_{n, \varepsilon}(x,t) \, d\mathcal V(x)}_{I_2(t)} \, .
\end{aligned}
\end{equation}
By reasoning similarly to the proof of Lemma \ref{Lemma11} and exploiting Lemma \ref{Lemma0} with $ f=u_2(t+h)-u_1(t) $ and \eqref{e78}--\eqref{e77}, we obtain the following estimates:
\begin{equation}\label{e149-ter}
|I_1(t)| \le \bar{C} \, \| W(t) \|_\infty \, \| a_n(t) + \varepsilon \|_\infty \, \int_{B_{2R}\setminus B_R} \varphi_{n, \varepsilon}(x,t) \, d\mathcal V(x) \, ,
\end{equation}
\begin{equation}\label{e149-quater}
|I_2(t)| \le \frac{C \| a_n(t) + \varepsilon \|_\infty}{2R} \left( \int_{B_{2R}\setminus B_R} \left|\nabla W(x,t) \right|^2 \, d\mathcal V(x) + \int_{B_{2R}\setminus B_R} \varphi_{n, \varepsilon}^2(x,t) \, d\mathcal V(x) \right) .
\end{equation}
Integrating \eqref{e149}, \eqref{e149-bis}, \eqref{e149-ter}, \eqref{e149-quater} between any $ t_{k+1} < t_\ast < t^\ast < t_{k} $, noting that $ \nabla W \in [L^2(M \times (t_\ast,t^\ast) )]^N $ and $ \varphi_{n, \varepsilon} \in L^2(M \times (t_\ast,t^\ast) ) $, and letting $ R \to \infty $, we end up with
\begin{equation}\label{e147-last}
\begin{aligned}
&\int_M W(x,t^\ast) \varphi_{n,
\varepsilon}(x,t^\ast) \, d\mathcal V(x) - \int_M W(x,t_\ast) \varphi_{n,
\varepsilon}(x,t_\ast) \, d\mathcal V(x) \\  = & - \int_{t_\ast}^{t^\ast}\int_M \left[a_n(x,
\tau) + \varepsilon-a(x,\tau)\right] \Delta W(x, \tau)\varphi_{n,
\varepsilon}(x, \tau) \, d\mathcal V(x) d\tau \, .
\end{aligned}
\end{equation}
The validity of \eqref{e147} just follows from \eqref{e147-last}, since the r.h.s.~of \eqref{e147-last} is in $L^1((\tau, T))$ (e.g.~as a function of $ t^\ast $) for all $\tau\in (0, T)$.
\end{proof}

\begin{lem}\label{Lemma14}
Let $a_n$, $\varphi_{n, \varepsilon}$, $\psi$ be as in Lemma \ref{Lemma12}. Then
\begin{equation}\label{e152}
\begin{aligned}
& \int_M \varphi_{n,\varepsilon}(x,t) \phi(x) \, d\mathcal{V}(x) - \int_M \psi(x) \phi(x) \, d\mathcal{V}(x) \\
= & \int_M \Delta \phi (x) \left[ \int_t^T (a_n(x,\tau)
+\varepsilon) \varphi_{n,\varepsilon}(x,\tau) \, d\tau  \right]
d\mathcal{V}(x)  \quad \textrm{for all}\ t \in (0,T), \, \phi\in
C^\infty_c(M) \, .
\end{aligned}
\end{equation}
In particular,
\begin{equation}\label{e153}
\int_M \varphi_{n,\varepsilon}(x,t) \, d\mathcal{V}(x) = \int_M \psi(x)\, d\mathcal{V}(x) \quad \textrm{for all}\ t \in (0,T) \, .
\end{equation}
\end{lem}
\begin{proof}
The validity of \eqref{e152} just a consequence of \eqref{e141} plus the continuity of $ \varphi_{n,\varepsilon}(t) $ as a curve in $ L^2(M) $ (for instance).

In order to establish \eqref{e153}, let us plug $\phi=\xi_R$ in \eqref{e152}, with $\xi_R$ still
defined as in the proof of Lemma \ref{Lemma11}. Thanks to \eqref{e78} and \eqref{e77}, we obtain:
\begin{equation}\label{e160}
\begin{aligned}
&\left|\int_M \varphi_{n, \varepsilon}(x,t)\xi_R(x)\,d\mathcal
V(x) - \int_M \psi(x)\xi_R(x)\, d\mathcal V(x) \right|\\
\leq & \bar C \, \|a_n + \varepsilon\|_{L^\infty(M\times (t,
T))}\int_{B_{2R}\setminus B_R}\int_t^T \varphi_{n,
\varepsilon}(x,\tau) \, d\mathcal V(x) d\tau\,.
\end{aligned}
\end{equation}
Since \eqref{e142} trivially implies $\varphi_{n, \varepsilon}\in
L^1(M\times (0, T))$, by letting $R\to\infty$ in \eqref{e160} we
deduce \eqref{e153}\,.
\end{proof}

\begin{lem}\label{Lemma15}
Let $a_n$, $\varphi_{n, \varepsilon}$, $\psi$ be as in Lemma
\ref{Lemma12}. We denote as $\Phi_{n,\varepsilon}(t)$ the potential of
$\varphi_{n, \varepsilon}(t)$, that is
\[\Phi_{n,\varepsilon}(x,t):=\mathcal G^{\varphi_{n,\varepsilon}(t)}(x) \, . \]
Then $ \nabla \Phi_{n, \varepsilon}(t) \in [L^2(M)]^N $ and the identity
\begin{equation}\label{e154}
\|\nabla \mathcal G^{\psi}\|_{2}^2\,=\,\|\nabla
\Phi_{n,\varepsilon}(t)\|_2^2\,+\,2\int_t^T\int_M[a_n(x,\tau)+\varepsilon]\varphi_{n,
\varepsilon}^2(x,\tau)\, d\mathcal V(x) d\tau
\end{equation}
holds for all $t\in (0, T]\,.$ 
\end{lem}
\begin{proof}
Since $ \Phi_{n, \varepsilon}(t) $ is the potential of $ \varphi_{n,\varepsilon}(t) $, which belongs to $ L^1(M) \cap L^\infty(M) $ (recall \eqref{e142}), thanks to Lemma \ref{Lemma0} we have that $ \Phi_{n, \varepsilon}(t) \in L^p(M) $ for all $ p \in (N/(N-2),\infty] $, $\nabla \Phi_{n, \varepsilon}(t) \in [L^2(M)]^N$ and
\begin{equation}\label{e161}
\|\nabla \Phi_{n, \varepsilon}(t)\|_2^2=\int_M \Phi_{n,
\varepsilon}(x,t) \varphi_{n, \varepsilon}(x,t)\, d\mathcal
V(x)\,.
\end{equation}
Furthermore, one can show that $ \Phi_{n, \varepsilon}(t) $ is an absolutely continuous curve in $  L^p(M) $ for all $ p \in (N/(N-2),\infty) $, satisfying the following differential equation:
\begin{equation}\label{e163}
(\Phi_{n,
\varepsilon})_t(x,t)\,=\,[a_n(x,t)+\varepsilon]\varphi_{n,
\varepsilon}(x,t)\quad \textrm{for a.e.}\ (x,t)\in M\times (0,
T)\,.
\end{equation}
This can be established exactly as we did for \eqref{e130a}. Taking
advantage of \eqref{e138}, \eqref{e161} and \eqref{e163}, we then deduce that
\begin{equation}\label{e164}
\frac{d}{dt}\|\nabla\Phi_{n, \varepsilon}(t)\|_2^2 = \int_M[a_n(x,t)+\varepsilon]\varphi^2_{n, \varepsilon}(x,t)\,
d\mathcal V(x) - \int_M \Phi_{n,\varepsilon}(x,t) \Delta[(a_n+\varepsilon)\varphi_{n, \varepsilon}](x,t)\,
d\mathcal V(x)
\end{equation}
for a.e.~$ t\in (0,T) $. Nevertheless, by exploiting the integrability properties of $ \Phi_{n,\varepsilon} $, $ \nabla \Phi_{n,\varepsilon} $, $ \varphi_{n, \varepsilon} $ and $ \Delta[(a_n+\varepsilon)\varphi_{n, \varepsilon}] $, the last term in the r.h.s.~of \eqref{e164} can be integrated by parts (through the same cut-off techniques we used in the proof of Lemma \ref{Lemma13}), which yields
\begin{equation}\label{e164-bis}
\frac{d}{dt}\|\nabla\Phi_{n, \varepsilon}(t)\|_2^2 = 2 \int_M[a_n(x,t)+\varepsilon]\varphi^2_{n, \varepsilon}(x,t)\,
d\mathcal V(x) \quad \textrm{for a.e.}\ t\in (0,T) \, .
\end{equation}
Since the r.h.s.~of \eqref{e164-bis} is in $L^1((\tau, T))$ for each
$\tau\in (0, T)$, $t \mapsto \|\nabla \Phi_{n, \varepsilon}(t)\|_2^2$ is
continuous on $(0, T]$ and absolutely continuous on every
$(t_{k+1}, t_k)$, the conclusion follows by integrating \eqref{e164-bis} over $(t, T)$.
\end{proof}

\subsection{Taking the limit of $\varphi_{n, \varepsilon}$ as $n\to \infty$}\label{limit-n}
This section is devoted to show that, for each fixed $\varepsilon>0$, the sequence $\{\varphi_{n, \varepsilon}\}$ converges in a suitable sense to
a limit function $\varphi_\varepsilon$, as $n\to \infty$. Moreover, such $\varphi_\varepsilon$ inherits some fundamental integrability properties from
$\{\varphi_{n, \varepsilon}\}$.

\begin{lem}\label{Lemma17}
Let $u_1, u_2$ be any two solutions of problem \eqref{e64}, taking
on the same initial datum $\mu\in \mathcal M_F^+(M)$. Let $W$ be
defined as in \eqref{e146}, $a$ as in \eqref{e135}, $\varphi_{n,
\varepsilon}$, $\psi$ as in Lemma \ref{Lemma12}. Then, up to
subsequences, $\{\varphi_{n, \varepsilon}\}$ converges weakly in
$L^2(M\times (\tau, T))$ (for each $\tau\in (0, T)$), as $n\to
\infty$, to a suitable nonnegative function $\varphi_\varepsilon$.
Moreover, there hold
\begin{equation}\label{e156}
\begin{aligned}
& \int_M \varphi_{\varepsilon}(x,t) \phi(x) \, d\mathcal{V}(x) - \int_M \psi(x) \phi(x) \, d\mathcal{V}(x) \\
= & \int_M \Delta \phi (x) \left[ \int_t^T (a(x,\tau)
+\varepsilon) \varphi_{\varepsilon}(x,\tau) \, d\tau  \right]
d\mathcal{V}(x)  \quad \textrm{for a.e.}\ t \in
(0,T)\,, \ \textrm{for any}\,\, \phi\in C^\infty_c(M) \,,
\end{aligned}
\end{equation}
\begin{equation}\label{e157}
\int_M \varphi_{\varepsilon}(x,t) \, d\mathcal{V}(x) = \int_M
\psi(x)\, d\mathcal{V}(x) \quad \textrm{for a.e.}\ t \in (0,T) \,,
\end{equation}
and
\begin{equation}\label{e158}
\begin{aligned}
& \left|\int_M W(x, T) \psi(x)\, d\mathcal V(x)-\int_M W(x,t)
\varphi_\varepsilon(x,t)\, d\mathcal V(x) \right|\\
 \leq & \varepsilon (T-t) \|\psi\|_1 \|u_2(\cdot +h)-u_1(\cdot)
\|_{L^\infty(M\times (t, T))}\quad \textrm{for a.e.}\,\, t\in
(0,T)\,.
\end{aligned}
\end{equation}
\end{lem}
\begin{proof}
From \eqref{e154} we infer that $\{\varphi_{n, \varepsilon}\}$
converges weakly (up to subsequences) in $L^2(M\times (\tau, T))$
(for each $\tau\in (0, T)$) to some $\varphi_{\varepsilon}\,.$
Moreover, thanks to \eqref{e142}, for every $t\in (0,T)$ there
exists a subsequence (which a priori depends on $t$) such that
\begin{equation}\label{e167}
\int_M \varphi_{n, \varepsilon}(t) \phi(x) \, d\mathcal V(x)
\to \, \int_M \phi(x) \, d\nu_\varepsilon^t(x)\quad \textrm{for any}\;\,\phi\in C_c(M)\,,
\end{equation}
for some $\nu_\varepsilon^t\in \mathcal M_F^+(M)$. In fact we have
that $d\nu_\varepsilon^t=\varphi_\varepsilon(t)d\mathcal V$ for a.e. $t\in (0,T)$. In order to show that, let $t\in (0,T)$ be a Lebesgue point for
$\varphi_\varepsilon(t)$ as a curve in $L^1((\tau, T); L^2(M))$. Take any $\phi\in C^\infty_c(M)$. Since for each $\tau\in (0,T)$ the sequence $\{\|a_n+\varepsilon\|_{L^\infty(M\times (\tau, T))}\}$ is bounded, in view of \eqref{e142} and \eqref{e152} we obtain
\begin{equation}\label{e165}
\begin{aligned}
&\left|\int_t^{t+\delta}\int_M \varphi_{n, \varepsilon}(x,
\tau)\phi(x) \, d\mathcal V(x) d\tau - \int_t^{t+\delta}\int_M
\varphi_{n, \varepsilon}(x,t)\phi(x) \, d\mathcal V(x) d\tau
\right| \\ \leq & \int_t^{t+\delta}
C(\tau-t)\|\psi\|_1\|\Delta\phi\|_\infty\, d\tau = \frac{\delta^2
C}{2}\|\psi\|_1 \|\Delta \phi\|_\infty\,
\end{aligned}
\end{equation}
for all $0<\delta<T-t$, for some positive constant $C$ independent
of $n$, $ \delta$. By letting $n\to \infty$ in \eqref{e165} (up to
subsequences) we get
\begin{equation}\label{e166}
\left|\int_t^{t+\delta}\int_M \varphi_\varepsilon(x, \tau)
\phi(x)\, d\mathcal V(x) d\tau -\delta \int_M \phi(x) \, d\nu^t_\varepsilon(x)
\right| \leq \frac{\delta^2 C}{2}\|\psi\|_1 \|\Delta
\phi\|_\infty\,.
\end{equation}
Upon dividing \eqref{e166} by $\delta$ and then letting $\delta \to
0^+$ we deduce that
\[\int_M \varphi_\varepsilon(x,t) \phi(x) \, d\mathcal V(x)\,=\, \int_M \phi(x) \, d\nu_\varepsilon^t(x)\,,\]
namely $\varphi_\varepsilon(t)d\mathcal V=d\nu_\varepsilon^t\,.$ Therefore,
the validity of \eqref{e156} easily follows by passing to the
limit in \eqref{e152} as $n\to \infty$, also in view of the convergence
properties of the sequence $\{a_n\}\,.$

\smallskip

Identity \eqref{e157} and estimate \eqref{e158} can be obtained along the lines of proof of \cite[Lemma 5.7]{GMPmu}: we only mention that \eqref{e157} follows by passing to the limit in \eqref{e160} and \eqref{e158} follows by passing to the limit in \eqref{e147}, which is feasible since $ W(t) \in C_b(M)$ and, thanks to \eqref{e157}, \eqref{e167} also holds for any $ \phi \in C_b(M) $.
%
%
\end{proof}

\subsection{Taking the limit of $\varphi_{\varepsilon}$ as $\varepsilon\to 0$ and proof of Theorem \ref{thmuni}}\label{limit-eps}
In order to prove Theorem \ref{thmuni}, we need to exploit the properties of the functions $\varphi_\varepsilon$ provided by
Lemma \ref{Lemma17}, and then let $\varepsilon\to 0$.

\begin{proof}[Proof of Theorem \ref{thmuni}]
Let $\Phi_\varepsilon(t)$ be the potential of
$\varphi_\varepsilon(t)$, that is $ \Phi_\varepsilon(x,t) = \mathcal G^{\varphi_\varepsilon(t)}(x) $. In view of \eqref{e156}, there follows
\begin{equation}\label{e170}
\mathcal G^{\psi}(x)- \Phi_\varepsilon(x, t) = \int_t^T [a(x,\tau)+\varepsilon]\varphi_\varepsilon(x, \tau)\, d\tau\geq
0\quad \textrm{for a.e.}\ (x,t)\in M\times(0, T)\,.
\end{equation}
This can be established as we did in the proof of \eqref{e130}: it is enough to plug the test function $ \xi_R \mathcal{G}^\phi $ in \eqref{e156} and let $ R \to \infty $, exploiting the fact that $ \mathcal G^\phi \in L^\infty(M) $, $ \nabla \mathcal{G}^\phi \in [L^2(M)]^N $, $ a \in L^\infty(M\times(t,T)) $ and $ \varphi_\varepsilon \in L^1(M\times(t,T)) \cap L^2(M\times(t,T)) $.

So, in particular,
\begin{equation}\label{e171}
0\leq \Phi_\varepsilon(x, t_1) \leq \Phi_\varepsilon(x, t_2) \leq
\mathcal G^\psi(x)\quad \textrm{for a.e.}\ x\in M \, , \
0<t_1<t_2<T\,.
\end{equation}
We now let $\varepsilon\to 0$. In view of \eqref{e171} it follows
that $\{\Phi_\varepsilon\}$ is bounded in $L^p(M\times (0, T))$ for
any $ p \in (N/(N-2) , \infty]$. In particular, there exists a sequence
$\{\Phi_{\varepsilon_n}\}$ that converges weakly in $L^p(M\times
(0, T))$ to some $\Phi\in L^p(M\times (0,T))$. As a consequence, in view of \eqref{e170} and
\eqref{e171}, by arguments similar to those used in the beginning
of the proof of Lemma \ref{Lemma17}, we can deduce that
$\{\Phi_{\varepsilon_n}(t)\}$ converges weakly in $L^p(M)$ to
$\Phi(t) $ for a.e.~$t\in (0, T)\,.$ Thanks to the boundedness of
$\{\varphi_{\varepsilon_n}(t)\}$ in $L^1(M)$ (recall \eqref{e157}),
for a.e.~$t\in (0, T)$ there exists a subsequence
$\{\varepsilon_{n_k}\}$ (a priori depending on $t$) such that
\begin{equation}\label{e172}
\int_M \varphi_{\varepsilon_{n_k}}(x, t)\phi(x)\, d\mathcal
V(x)\to \int_M \phi(x)\, d\nu^t(x)\quad \textrm{as}\ k\to
\infty \, , \ \textrm{for any}\ \phi\in C_0(M)\, ,
\end{equation}
for some $\nu^t\in \mathcal M_F^+(M)$. Hence, for a.e.~$t\in
(0, T)$, for any $\phi\in C_c(M)$, there holds
\begin{equation}\label{e173}
\begin{aligned}
\int_M \mathcal G^{\nu^t} \phi(x)\, d\mathcal V(x) \,=\,\lim_{k\to
\infty}\, \int_M \varphi_{n_k}(x,t) \mathcal G^\phi(x)\,d\mathcal
V(x) =& \lim_{k\to \infty} \int_M \Phi_{n_k}(x,t)
\phi(x)\,d\mathcal V(x)\\= &\int_M \Phi(x, t) \phi(x)\, d\mathcal
V(x) \,.
\end{aligned}
\end{equation}
Due to \eqref{e173} and Lemma \ref{Lemma-2}, we infer
that $\nu^t$ is independent of the particular subsequence, so that
\eqref{e172} holds along the whole sequence $\{\varepsilon_n\}$
and $\Phi(t)$ is the potential of $\nu^t$. Moreover, in view of
\eqref{e171} and of the convergence properties of $\{\Phi_{\varepsilon_n}\}$, we get
\begin{equation}\label{e174}
0\leq \Phi(x, t_1) \leq \Phi(x, t_2) \leq \mathcal G^\psi(x)\quad
\textrm{for a.e.}\ x\in M \, , \ 0<t_1<t_2<T\,.
\end{equation}
We now aim at proving that \eqref{e172} holds for any $\phi\in
C_b(M)\,.$ To this end, note that since \eqref{e157} holds and
$a\in L^\infty(M\times (\tau, T))$ for each $\tau\in (0, T)$, we
have that, up to subsequences,
\begin{equation*}\label{e175}
\int_M \phi(x) \left\{\int_t^T [a(x,
\tau)+\varepsilon_n]\varphi_{\varepsilon_n}(x,\tau) \, d\tau
\right\}d\mathcal V(x)\to \int_M \phi(x)\, d\sigma^{t, T}(x) \quad
\textrm{as}\ n \to \infty\,,
\end{equation*}
for a.e.~$t\in (0, T),$ for any $\phi\in C_c(M)$, where $ \sigma^{t, T} $ is a suitable element of
$ \mathcal M^+_F(M)$. We can therefore pass to
the limit as $n\to \infty$ in \eqref{e156} (with $\varepsilon=\varepsilon_n$) to get
\begin{equation}\label{e176}
\begin{aligned}
\int_M \phi(x) \, d\nu^t(x) - \int_M \psi(x) \phi(x) \, d\mathcal{V}(x) = \int_M \Delta \phi (x) \, d\sigma^{t, T}(x)
\end{aligned}
\end{equation}
for a.e.~$t \in (0,T)$, for any $\phi\in C^\infty_c(M)$. Now let us plug $\phi=\xi_R$ in \eqref{e176}, with $\xi_R$ defined as in the proof of Lemma \ref{Lemma11}. Thanks to \eqref{e78} and \eqref{e77}, we obtain
\begin{equation}\label{e177}
\begin{aligned}
\left| \int_M \xi_R(x) \, d\nu^t(x) - \int_M \psi(x) \xi_R(x) \, d\mathcal{V}(x)\right|
\leq \bar C \int_{B_{2R}\setminus B_R} \, d\sigma^{t, T}(x)
\quad \textrm{for a.e.}\ t \in (0,T)\,.
\end{aligned}
\end{equation}
Since $\sigma^{t, T}$ is a positive finite measure, by letting $R\to \infty$ in \eqref{e177} we get
\begin{equation}\label{e178}
\int_M d\nu^t(x)\,=\, \int_M \psi(x) \, d\mathcal V(x)\,.
\end{equation}
Due to \eqref{e157}, \eqref{e172}, \eqref{e178} and
\cite[Proposition 1.80]{AFP} we then deduce that
\begin{equation}\label{e179}
\int_M \varphi_{\varepsilon_{n}}(x, t)\phi(x)\, d\mathcal V(x)\to
\int_M \phi(x)\, d\nu^t(x)\quad \textrm{as}\,\, n\to \infty\,\,
\textrm{for any}\;\, \phi\in C_b(M)\,.
\end{equation}
As a consequence of \eqref{e179} we get
\begin{equation}\label{e139}
\int_M W(x, T) \psi(x) \, d\mathcal V(x) = \int_M W(x,t)\, d\nu^t(x) \ \ \ {\rm for\ a.e.}\ t\in(0,T) \, ,
\end{equation}
by passing to the limit as $\varepsilon=\varepsilon_n \to 0$ in \eqref{e158} (recall that, from Lemma \ref{Lemma0}, $W(t)\in C_b(M)$).

Since for a.e.~$ 0< t_\ast<t^\ast<T $ we have that $ \Phi(x,t_\ast) \le \Phi(x,t^\ast)  $ for a.e.~$ x \in M $ (see \eqref{e174}), it is direct to show that the curve $\nu^t$ can be extended to \emph{every} $t\in (0, T]$ so that it
still satisfies \eqref{e174}, \eqref{e178} and \eqref{e139}\,.
Hence, in view of Lemma \ref{Lemma11} and \eqref{e139}, we deduce that
\begin{equation}\label{e180}
\int_M W(x,T) \psi(x)\, d\mathcal V(x)\,\leq\, \int_M [\mathcal
G_2(x, h)-\mathcal G_1(x, t_0)]\, d\nu^t(x) \quad \textrm{for
all}\,\, 0<t<t_0<T\,.
\end{equation}
Thanks to \eqref{e174} and \eqref{e178}, it is straightforward to
check that there exists $\nu_0\in \mathcal M_F^+(M)$ such that
\begin{equation*}\label{e181}
\lim_{t\to 0} \int_M \phi(x) \, d\nu^t(x)\,=\, \int_M \phi(x)\,
d\nu_0 \quad \textrm{for any}\ \phi\in C_c(M)
\end{equation*}
and
\begin{equation*}\label{e182}
\mathcal G^{\nu_0}(x)\,=\,\lim_{t\to 0} \Phi(x, t):=\Phi_0(x)
\quad \textrm{for a.e.}\ x\in M\,.
\end{equation*}
Thus, by dominated convergence and Tonelli's theorem,
\begin{equation*}\label{e183}
\begin{aligned}
\lim_{t\to 0} \int_M \mathcal G_1(x, t_0)\, d\nu^t(x)=
\lim_{t\to 0} \int_M u_1(x, t_0) \Phi(x,t)\, d\mathcal V(x) = & \int_M u_1(x, t_0)\Phi_0(x)\, d\mathcal V(x) \\
= & \int_M \mathcal G_1(x, t_0)\, d\nu_0(x)\,.
\end{aligned}
\end{equation*}
We can similarly prove that
\begin{equation*}\label{e184}
\lim_{t\to 0} \int_M \mathcal G_2(x, h)\, d\nu^t(x)\,=\, \int_M
\mathcal G_2(x, h)\, d\nu_0(x)\,.
\end{equation*}
Hence, passing to the limit as $t\to 0$ in \eqref{e180} we infer
that
\begin{equation}\label{e185}
\int_M W(x,T) \psi(x)\, d\mathcal V(x)\,\leq\, \int_M [\mathcal
G_2(x, h)-\mathcal G_1(x, t_0)]\, d\nu_0(x) \quad \textrm{for
all}\,\, 0<t_0<T\,.
\end{equation}
Letting $t_0\to 0$ in \eqref{e185}, by monotone convergence and in view of Lemma \ref{Lemma11} we find
\begin{equation}\label{e186}
\int_M W(x,T) \psi(x)\, d\mathcal V(x)\,\leq\, \int_M [\mathcal
G_2(x, h)-\mathcal G^\mu]\, d\nu_0(x)\leq 0\,.
\end{equation}
Since $h>0$, $ T>0$ and $\psi\in C^\infty_c(M)$ (with $ \psi\geq 0$) are
arbitrary, from \eqref{e186} there follows $\mathcal G_1 \geq
\mathcal G_2$. Interchanging the role of $u_1$ and $u_2$, we also get $\mathcal G_1 \leq
\mathcal G_2$, so that $\mathcal G_1=\mathcal G_2$ and $u_1=u_2$ in view of Lemma \ref{Lemma-2}.
\end{proof}

\begin{oss}\label{ossuni}\rm

As a consequence of the above method of proof, we point out that Theorem \ref{thmuni} still holds under the weaker assumption that \eqref{e68} is satisfied for any $\phi\in C_c(M)$.

However, with respect to the existence counterpart, in order to prove Theorem \ref{thmuni} we require some additional hypotheses. First of all, we assume \eqref{H}-\textrm{(ii)}: this is essential to provide a cut-off family $ \xi_R $ satisfying \eqref{e77}, which is the main tool we exploit to justify all the integration by parts, as well as the conservation of mass \eqref{e178}. The positivity of the initial datum, namely the fact that $  \mu \in M^+_F(M) $, is crucial for the validity of \eqref{e131} for \emph{every} $ x \in M $ and for the monotonicity of $ \mathcal{G} $ as a function of $t$: both properties are deeply exploited in the final part of the proof of Theorem \ref{thmuni}.

Finally, in Remark \ref{supercritical} we explained how to recover existence in the range $ \frac{N-2}{N} < m < 1 $. As for uniqueness, there are two main issues. Indeed, a priori nothing guarantees that $ u^m(t) \in L^1(M) $ for positive times: this prevents us from proving that the remainder integrals $I_1$ and $I_2$ in the proof of Lemma \ref{Lemma11} vanish as $R \to \infty$.
On the other hand, the function $a$ in \eqref{e135} is no more bounded for positive times, a crucial property that we exploit throughout the proof of Theorem \ref{thmuni}.
\end{oss}

\subsection{Proof of existence and uniqueness of the initial trace}
First of all let us note that, under the assumptions of Theorem \ref{thmtracce}, the proof of Lemma \ref{Lemma11} works without further issues down to the proof of identity \eqref{e130}. Moreover, by combining the latter with the smoothing effect \eqref{e70-pre} and proceeding as in the proof of Theorem \ref{thmexi}, we end up with the estimate
\begin{equation}\label{e125-trace}
\left| \int_M \mathcal{G}(x, t_2) \phi(x) \, d\mathcal V(x) - \int_M \mathcal{G}(x, t_1) \phi(x) \, d \mathcal V(x) \right| \leq \|\phi \|_{\infty} \, K^{m-1} \left\| u \right\|_{L^\infty((0,\infty);L^1(M))}^{1+\beta(m-1)} \, \int_{t_1}^{t_2}t^{-\alpha(m-1)} \, dt \, ,
\end{equation}
where we denote again by $ \mathcal{G}(t) $ the potential of $u(t)$. Furthermore, given any $ t_2>t_1>0 $ and any $ R \ge 1 $, by plugging in \eqref{e67} the test function
\[
\varphi(x,t)= \xi_R(x) \left[\theta_\varrho^{t_1}(t) - \theta_\varrho^{t_2}(t) \right] \quad \textrm{for all}\ x \in M \, , \ t>0
\]
($ \theta_\varrho^{t_\cdot} $ is defined as in the proof of Proposition \ref{pro: cons} and $ \xi_R $ as in the proof of Lemma \ref{Lemma11}), integrating by parts, letting $ \varrho \to 0 $ and using \eqref{e77}, we obtain:
\begin{equation}\label{eq: cons-mass-trace}
\left| \int_M u(x, t_2) \xi_R(x) \, d\mathcal V(x) - \int_M u(x, t_1) \xi_R(x) \, d\mathcal V(x) \right| \le  \bar C \int_{t_1}^{t_2} \int_{B_{2R}\setminus B_R}|u(x,t)|^m \, d\mathcal V(x) dt \, .
\end{equation}
In addition, \eqref{e65} and the smoothing effect \eqref{e70-pre} ensure that
\begin{equation}\label{eq: cons-mass-trace-bis}
\int_{0}^{t_2} \int_{M} |u(x,t)|^m \, d\mathcal V(x) dt \le \frac{K^{m-1}}{1-\alpha(m-1)} \left\| u \right\|_{L^\infty((0,\infty);L^1(M))}^{1+\beta(m-1)} t_2^{1-\alpha(m-1)} < \infty \, .
\end{equation}
Having established \eqref{e125-trace}--\eqref{eq: cons-mass-trace-bis} for general weak solutions to the differential equation in \eqref{e64}, we are in position to prove Theorem \ref{thmtracce}.
\begin{proof}[Proof of Theorem \ref{thmtracce}]
In view of \eqref{e125-trace} we can infer that the family $\{\mathcal{G}(t)\}$ is Cauchy in $L^1_{\rm loc}(M)$ as $ t \to 0 $, hence there exists a function $ \mathcal{G}_0 \in L^1_{\rm loc}(M)$ such that $\mathcal G(t)\to \mathcal{G}_0$ as $t\to 0$ in $L^1_{\rm loc}(M)$. Moreover, the fact that $u\in L^\infty((0,\infty); L^1(M))$ implies that for every sequence $t_n\to 0$ there
exists $\mu\in \mathcal M_F(M)$ such that $\{u(t_n)\}$ converges vaguely to $\mu$ as $n\to \infty$, up to a subsequence (recall Definition \ref{def1}). On the other hand, the convergence of $\{\mathcal{G}(t_n)\}$ to $\mathcal{G}_0$ in $L^1_{\rm loc}(M)$ implies $\mathcal G^{\mu}=\mathcal{G}_0$, so that by Lemma \ref{Lemma-2} the measure $\mu$ does not depend on the sequence $\{t_n\}$, and \eqref{e68} holds for all $ \phi \in C_c(M) $. In order to prove that \eqref{e68} also holds for constant functions, we can exploit \eqref{eq: cons-mass-trace}: by letting $ t_1 \to 0 $ and using the vague convergence of $ \{ u(t_1) \} $ to $ \mu $ we end up with
\begin{equation*}
\left| \int_M u(x, t_2) \xi_R(x) \, d\mathcal V(x) - \int_M \xi_R(x) \, d\mu(x) \right| \le \bar C \int_{0}^{t_2} \int_{B_{2R}\setminus B_R}|u(x,t)|^m \, d\mathcal V(x) dt \, .
\end{equation*}
We then let $ R \to \infty $: thanks to \eqref{eq: cons-mass-trace-bis} we obtain
\begin{equation*}\label{eq: cons-mass-trace-limit}
\left| \int_M u(x, t_2) \, d\mathcal V(x) - \int_M d\mu(x) \right| \le 0 \, ,
\end{equation*}
namely the conservation of mass or, equivalently, the fact that \eqref{e68} holds for $ \phi $ equal to any constant. In the case where $ u \ge 0 $, the last assertion of the theorem is just a consequence of \cite[Proposition 1.80]{AFP}, since for positive measures the vague convergence plus the convergence of the measures is equivalent to convergence in the dual space of $ C_b(M) $.
%
%
\end{proof}

\end{document}